\documentclass[11pt]{amsart}
\title{Resurrection axioms and uplifting cardinals}

\author[Hamkins]{Joel David Hamkins}
 \address[J. D. Hamkins]
        {Mathematics, Philosophy, Computer Science,
          The Graduate Center of The City University of New York,
          365 Fifth Avenue, New York, NY 10016
          \&
          Mathematics,
          College of Staten Island of CUNY,
          Staten Island, NY 10314}
\email{jhamkins@gc.cuny.edu}
\urladdr{http://jdh.hamkins.org}

\author[Johnstone]{Thomas A. Johnstone}
 \address[T. A. Johnstone]
         {Mathematics,
          New York City College of Technology of CUNY, Brooklyn, 
          NY 11201}
 \email{tjohnstone@citytech.cuny.edu}
 \urladdr{http://websupport1.citytech.cuny.edu/faculty/tjohnstone/}

\thanks{\tiny The authors would like to apologize for the long delay in bringing this work to completion; we've studied the resurrection idea since 2000, and the main equiconsistency with uplifting cardinals was proved at Bedlewo in 2007; and we've given numerous talks on it since then. The research of the first author has been supported in part by NSF grant DMS-0800762, PSC-CUNY grant 64732-00-42, Simons Foundation grant 209252, the Netherlands Organization for Scientific Research NWO \textsl{bezoekersbeurs} \textsf{B 62-619 2006/00782/IB}, and he is grateful to the Institute of Logic, Language and Computation at Universiteit van Amsterdam for the support of a Visiting Professorship during his sabbatical there in 2007, where the two authors worked together. The research of the second author has been supported by a CUNY Scholar Incentive Award, PSC-CUNY research grants \#62803-00-40 and \#64682-00-42, and he is grateful to the Kurt \Godel\ Research Center at the University of Vienna for the support of his 2009-10 post-doctoral position there, funded in part by grants P20835-N13 and P21968-N13 from the FWF Austrian Science Fund. The authors would like to thank the referee for helpful comments and suggestions that have been incorporated into this article. Commentary concerning this article can be made at \url{http://jdh.hamkins.org/resurrection-axioms-and-uplifting-cardinals}.}

\usepackage{latexsym,amsfonts,amsmath,amssymb}
\usepackage[hidelinks]{hyperref}
\usepackage{verbatim}
\newtheorem{theorem}{Theorem}
\newtheorem{maintheorem}[theorem]{Main Theorem}

\newtheorem{lemma}[theorem]{Lemma}

\newtheorem{question}[theorem]{Question}

\newtheorem{observation}[theorem]{Observation}

\newtheorem{fact}[theorem]{Fact}
\newtheorem{definition}[theorem]{Definition}

\newtheorem{maindefinition}[theorem]{Main Definition}

\newcommand{\QED}{\end{proof}}

\def\proclaim[#1]{{\bf #1}}
\def\BF#1.{{\bf #1.}}

\newcommand{\Godel}{G\"odel}

\newcommand{\Levy}{L\'{e}vy}
\newcommand{\Lowenheim}{L\"owenheim}
\newcommand{\Velickovic}{Veli\v ckovi\'c}
\newcommand{\B}{{\mathbb B}}

\renewcommand{\P}{{\mathbb P}}
\newcommand{\Q}{{\mathbb Q}}

\newcommand{\R}{{\mathbb R}}

\newcommand{\continuum}{\mathfrak{c}}

\newcommand{\Ptail}{{\P_{\!\scriptscriptstyle\rm tail}}}

\newcommand{\Vbar}{{\overline{V}}}

\newcommand{\Qdot}{{\dot\Q}}

\newcommand{\Rdot}{{\dot\R}}

\newcommand{\one}{\mathop{1\hskip-3pt {\rm l}}}
\newcommand{\from}{\mathbin{\vbox{\baselineskip=2pt\lineskiplimit=0pt
                         \hbox{.}\hbox{.}\hbox{.}}}}

\newcommand{\of}{\subseteq}
\newcommand{\ofnoteq}{\subsetneq}

\newcommand{\elesub}{\prec}

\newcommand{\dom}{\mathop{\rm dom}}

\newcommand{\ran}{\mathop{\rm ran}}

\newcommand{\Add}{\mathop{\rm Add}}

\newcommand{\Coll}{\mathop{\rm Coll}}

\newcommand{\Con}{\mathop{{\rm Con}}}
\newcommand{\image}{\mathbin{\hbox{\tt\char'42}}}
\newcommand{\plus}{{+}}

\newcommand{\restrict}{\upharpoonright} 
\newcommand{\satisfies}{\models}
\newcommand{\forces}{\Vdash}

\newcommand{\intersect}{\cap}

\newcommand{\LaverDiamond}{\mathop{\hbox{\line(0,1){10}\line(1,0){8}\line(-4,5){8}}\hskip 1pt}\nolimits}
\newcommand{\LD}{\LaverDiamond}

\newcommand{\LDuplift}{\LD^{\hbox{\!\!\tiny uplift}}}

\newcommand{\smalllt}{\mathrel{\mathchoice{\raise2pt\hbox{$\scriptstyle<$}}{\raise1pt\hbox{$\scriptstyle<$}}{\raise0pt\hbox{$\scriptscriptstyle<$}}{\scriptscriptstyle<}}}
\newcommand{\smallleq}{\mathrel{\mathchoice{\raise2pt\hbox{$\scriptstyle\leq$}}{\raise1pt\hbox{$\scriptstyle\leq$}}{\raise1pt\hbox{$\scriptscriptstyle\leq$}}{\scriptscriptstyle\leq}}}
\newcommand{\lt}{\smalllt}

\newcommand{\ltkappa}{{{\smalllt}\kappa}}

\newcommand{\ltdelta}{{{\smalllt}\delta}}

\newcommand{\boolval}[1]{\mathopen{\lbrack\!\lbrack}\,#1\,\mathclose{\rbrack\!\rbrack}}
\def\[#1]{\boolval{#1}}
\newbox\gnBoxA
\newdimen\gnCornerHgt
\setbox\gnBoxA=\hbox{\tiny$\ulcorner$}
\global\gnCornerHgt=\ht\gnBoxA
\newdimen\gnArgHgt
\def\gcode #1{%
\setbox\gnBoxA=\hbox{$#1$}%
\gnArgHgt=\ht\gnBoxA%
\ifnum     \gnArgHgt<\gnCornerHgt \gnArgHgt=0pt%
\else \advance \gnArgHgt by -\gnCornerHgt%
\fi \raise\gnArgHgt\hbox{\tiny$\ulcorner$} \box\gnBoxA %
\raise\gnArgHgt\hbox{\tiny$\urcorner$}}
\newcommand{\UnderTilde}[1]{{\setbox1=\hbox{$#1$}\baselineskip=0pt\vtop{\hbox{$#1$}\hbox to\wd1{\hfil$\sim$\hfil}}}{}}
\newcommand{\Undertilde}[1]{{\setbox1=\hbox{$#1$}\baselineskip=0pt\vtop{\hbox{$#1$}\hbox to\wd1{\hfil$\scriptstyle\sim$\hfil}}}{}}
\newcommand{\undertilde}[1]{{\setbox1=\hbox{$#1$}\baselineskip=0pt\vtop{\hbox{$#1$}\hbox to\wd1{\hfil$\scriptscriptstyle\sim$\hfil}}}{}}
\newcommand{\UnderdTilde}[1]{{\setbox1=\hbox{$#1$}\baselineskip=0pt\vtop{\hbox{$#1$}\hbox to\wd1{\hfil$\approx$\hfil}}}{}}
\newcommand{\Underdtilde}[1]{{\setbox1=\hbox{$#1$}\baselineskip=0pt\vtop{\hbox{$#1$}\hbox to\wd1{\hfil\scriptsize$\approx$\hfil}}}{}}

\newcommand{\st}{\mid}
\renewcommand{\th}{{\hbox{\scriptsize th}}}

\renewcommand{\iff}{\mathrel{\leftrightarrow}}

\newcommand{\minus}{\setminus}

\def\<#1>{\langle#1\rangle}

\newcommand{\cp}{\mathop{\rm cp}}

\newcommand{\ORD}{\mathop{{\rm Ord}}}
\newcommand{\Ord}{\mathop{{\rm Ord}}}

\newcommand{\ZFC}{{\rm ZFC}}

\newcommand{\CH}{{\rm CH}}

\newcommand{\MA}{{\rm MA}}

\newcommand{\RA}{{\rm RA}}

\newcommand{\MM}{{\rm MM}}
\newcommand{\BMM}{{\rm BMM}}

\newcommand{\PFA}{{\rm PFA}}

\newcommand{\BPFA}{{\rm BPFA}}
\newcommand{\SPFA}{{\rm SPFA}}
\newcommand{\BSPFA}{{\rm BSPFA}}

\newcommand{\MP}{{\rm MP}}
\newcommand{\HOD}{{\rm HOD}}

\newcommand{\ccc}{{{\rm ccc}}}

\newcommand{\cell}[1]{\boxit{\hbox to 17pt{\strut\hfil$#1$\hfil}}}
\newcommand{\head}[2]{\lower2pt\vbox{\hbox{\strut\footnotesize\it\hskip3pt#2}\boxit{\cell#1}}}
\newcommand{\boxit}[1]{\setbox4=\hbox{\kern2pt#1\kern2pt}\hbox{\vrule\vbox{\hrule\kern2pt\box4\kern2pt\hrule}\vrule}}
\newcommand{\Col}[3]{\hbox{\vbox{\baselineskip=0pt\parskip=0pt\cell#1\cell#2\cell#3}}}
\newcommand{\tapenames}{\raise 5pt\vbox to .7in{\hbox to .8in{\it\hfill input: \strut}\vfill\hbox to
.8in{\it\hfill scratch: \strut}\vfill\hbox to .8in{\it\hfill output: \strut}}}
\newcommand{\Head}[4]{\lower2pt\vbox{\hbox to25pt{\strut\footnotesize\it\hfill#4\hfill}\boxit{\Col#1#2#3}}}
\newcommand{\Dots}{\raise 5pt\vbox to .7in{\hbox{\ $\cdots$\strut}\vfill\hbox{\ $\cdots$\strut}\vfill\hbox{\
$\cdots$\strut}}}
\newcommand{\df}{\it} 
\newcommand{\RAall}{\RA(\text{all})}

\newcommand{\smallgt}{\mathrel{\mathchoice{\raise2pt\hbox{$\scriptstyle>$}}{\raise1pt\hbox{$\scriptstyle>$}}{\raise0pt\hbox{$\scriptscriptstyle>$}}{\scriptscriptstyle>}}}

\newcommand{\smallgeq}{\mathrel{\mathchoice{\raise2pt\hbox{$\scriptstyle\geq$}}{\raise1pt\hbox{$\scriptstyle\geq$}}{\raise1pt\hbox{$\scriptscriptstyle\geq$}}{\scriptscriptstyle\geq}}}

\newcommand{\wRA}{{\rm wRA}}

\newcommand{\proper}{\text{proper}}
\newcommand{\semiproper}{\text{semi-proper}}
\newcommand{\Hc}{H_\continuum}
\newcommand{\FA}{\rm FA}
\newcommand{\BFA}{\rm BFA}
\newcommand{\BAAFA}{\rm BAAFA}

\begin{document}

\begin{abstract}
 We introduce the resurrection axioms, a new class of forcing axioms, and the uplifting cardinals, a new large cardinal notion, and prove that various instances of the resurrection axioms are equiconsistent over \ZFC\ with the existence of an uplifting cardinal.
\end{abstract}

\maketitle

\section{Introduction}\label{S.Introduction}

Many classical forcing axioms can be viewed, at least informally, as the claim that the universe is existentially closed in its forcing extensions, for the axioms generally assert that certain kinds of filters, which could exist in a forcing extension $V[G]$, exist already in $V$. In several instances this informal perspective is realized more formally: Martin's axiom is equivalent to the assertion that $\Hc$ is existentially closed in all c.c.c.~forcing extensions of the universe, meaning that $\Hc\elesub_{\Sigma_1}V[G]$ for all such extensions; the bounded proper forcing axiom is equivalent to the assertion that $H_{\omega_2}$ is existentially closed in all proper forcing extensions, or $H_{\omega_2}\elesub_{\Sigma_1}V[G]$; and there are other similar instances.

In model theory, a submodel $M\of N$ is {\df existentially closed} in $N$ if existential assertions true in $N$ about parameters in $M$ are true already in $M$, that is, if $M$ is a $\Sigma_1$-elementary substructure of $N$, which we write as $M\elesub_{\Sigma_1} N$. Furthermore, in a general model-theoretic setting, existential closure is tightly connected with resurrection, the theme of this article.

\goodbreak
\begin{fact} \label{F:ExistClosIffResurr}
If $\mathcal{M}$ is a submodel of $\mathcal{N}$, then the following are equivalent.
\begin{enumerate}
 \item The model $\mathcal{M}$ is existentially closed in $\mathcal{N}$
\item $\mathcal{M}\of \mathcal{N}$ has resurrection. That is, there is a further extension $\mathcal{M}\of\mathcal{N}\of\mathcal{M}^+$ for which $\mathcal{M}\elesub\mathcal{M}^+$.
\end{enumerate}
\end{fact}

\begin{proof} If $\mathcal{M}$ is existentially closed in $\mathcal{N}$, then by compactness the elementary diagram of $\mathcal{M}$ is consistent with the atomic diagram of $\mathcal{N}$, and any model of this combined theory provides the desired $\mathcal{M}^+$. Conversely, resurrection implies existential closure, since any witness in $\mathcal{N}$ still exists in $\mathcal{M}^+$, and so $\mathcal{M}$ has witnesses by the elementarity of $\mathcal{M}\elesub\mathcal{M}^+$.
\end{proof}

We call this ``resurrection,'' because although certain truths in $\mathcal{M}$ may no longer hold in the extension $\mathcal{N}$, these truths are nevertheless revived in light of $\mathcal{M}\elesub\mathcal{M}^+$ in the further extension to $\mathcal{M}^+$. A difficulty arises when applying fact~\ref{F:ExistClosIffResurr} in the context of forcing axioms, however, where set theorists seek principally to understand how a given model $\mathcal{M}$ relates to its forcing extensions, rather than to the more arbitrary extensions $\mathcal{M}^+$ arising from the compactness theorem. The problem is that when one restricts the class of permitted models $\mathcal{M}^+$ in fact~\ref{F:ExistClosIffResurr}, the equivalence of (1) and (2) can break down. Nevertheless, the converse implication $(2)\to(1)$ always holds: every instance of resurrection implies the corresponding instance of existential closure. This key observation leads us to the main unifying theme of this article, the idea that {\bf resurrection may allow us to formulate more robust forcing axioms} than existential closure or than combinatorial assertions about filters and dense sets.

We shall therefore introduce in this paper a spectrum of new forcing axioms utilizing the resurrection concept. We shall analyze the relations between these new forcing axioms and the classical axioms, and in many cases find their exact large cardinal consistency strength. The main idea is to replace a forcing axiom expressible as
$$\forall\,\Q\quad \mathcal{M}\elesub_{\Sigma_1}\mathcal{M}^{V[g]}, \text{ whenever } g\of\Q \text{ is } V \text{-generic}$$
with an axiom asserting full elementarity in a further extension:
$$ \forall\,\Q\,\,\,\exists\, \Rdot\quad \mathcal{M}\elesub \mathcal{M}^{V[g*h]}, \text{ whenever } g*h\of\Q*\Rdot \text{ is } V \text{-generic},$$
where in each case the forcing notions $\Q$ and $\Rdot$ will be of a certain specified type appropriate for that forcing axiom. We had mentioned earlier that under \MA\ or \BPFA\ (which implies $\continuum=\omega_2$), the set $\Hc$ is existentially closed in $V[g]$ for all c.c.c.~or proper forcing $g\of \Q$, respectively, and the case of $\mathcal{M}=\< \Hc,{\in}>$ is central.

\goodbreak

\begin{maindefinition}\label{D.ResurrectionAxioms}
\rm Let $\Gamma$ be a fixed definable class of forcing notions.
\begin{enumerate}
 \item The {\df resurrection axiom} $\RA(\Gamma)$ is the assertion that for every forcing notion $\Q\in\Gamma$ there is further forcing $\Rdot$, with $\forces_\Q\Rdot\in\Gamma$, such that if $g*h\of\Q*\Rdot$ is $V$-generic, then $\Hc\elesub \Hc^{V[g*h]}$.
\smallskip
 \item The {\df weak resurrection axiom} $\wRA(\Gamma)$ is the assertion that for every $\Q\in\Gamma$ there is further forcing $\Rdot$, such that if
      $g*h\of\Q*\Rdot$ is $V$-generic, then $\Hc\elesub \Hc^{V[g*h]}$.
\end{enumerate}
\end{maindefinition}

The difference between the full axiom and the weak form is that the full axiom insists that the second step of forcing $\Rdot$ is also chosen from $\Gamma$, as interpreted in the extension $V[g]$, while the weak axiom drops this restriction. When determining whether $\forces_\Q\Rdot\in\Gamma$, we give $\Gamma$ the {\it de dicto} reading, meaning that we reinterpret $\Gamma$ in the extension $V[g]$, using the definition of $\Gamma$ in that model, so the question is whether $\Rdot_g\in\Gamma^{V[g]}$. Definition~\ref{D.ResurrectionAxioms} is a special case of the more general resurrection axiom  $\RA(\Gamma_0,\Gamma_1)$, asserting that for every $\Q\in\Gamma_0$ there is further forcing $\Rdot$ with $\forces_\Q\Rdot\in\Gamma_1$, such that whenever $g*h\of\Q*\Rdot$ is $V$-generic, then  $\Hc\elesub \Hc^{V[g*h]}$; but we shall not analyze this more general axiom here.

We shall consider instances $\RA(\Gamma)$ and $\wRA(\Gamma)$ for various natural classes $\Gamma$ of forcing notions, such as $\RA(\ccc)$ and $\wRA(\ccc)$ for the class of all c.c.c.~posets, $\RA(\proper)$ and $\wRA(\proper)$ for the class of all proper posets, and $\RAall$ for the class of all posets. Note that $\wRA(\text{all})$ is the same as $\RAall$. If $\Gamma$ is any class of forcing notions, then $\RAall$ implies $\wRA(\Gamma)$, and $\RA(\Gamma)$ implies $\wRA(\Gamma)$. Moreover, if $\Gamma_1\subseteq \Gamma_2$ are two classes of forcing notions, then $\wRA(\Gamma_2)$ implies $\wRA(\Gamma_1)$, but in general $\RA(\Gamma_2)$ need not imply $\RA(\Gamma_1)$.

Regarding the existential-closure remark in the opening sentence of this article, we note that the full set-theoretic universe $V$ is never actually existentially closed in any nontrivial extension $V\ofnoteq W$. The point is that $W$ will have some set $z$ not in $V$, and an $\in$-minimal such $z$ will have $z\of y$ for some $y\in V$, meaning that $W$ thinks there is a subset of $y$ not in $P(y)^V$, but $V$ does not; this is a $\Sigma_1$ assertion about $P(y)^V$ showing that $V\not\elesub_{\Sigma_1} W$. Similarly, in a nontrivial set-forcing extension $V\of V[g]$ for $V$-generic $g\of\Q$, where $\mathcal{D}$ is the collection of all dense subsets of $\Q$ in $V$, the universe $V[g]$ contains a filter that meets all elements of $\mathcal{D}$, but $V$ does not; and again this is a $\Sigma_1$ assertion about $\mathcal{D}$. If the forcing extension $V\of V[g]$ adds a new real, then the collection $H_{\continuum^+}$ is not existentially closed in $V[g]$, because the forcing extension $V[g]$ contains a subset of $\omega$ that is not an element of $\mathcal{P}(\omega)^V$, but $H_{\continuum^+}$ does not. So if our forcing notions will be able to add reals, then we will not have any existential closure for $H_\kappa$ when $\continuum<\kappa$, pointing again at the centrality of the case of $\Hc$. Meanwhile, if $\kappa$ is any uncountable cardinal, then the \Levy\ absoluteness theorem shows that $H_\kappa\elesub_{\Sigma_1}V$, and so in particular, $\Hc$ is always existentially closed in $V$. As intended, the resurrection axioms imply that this structure $\Hc$ remains existentially closed with respect to forcing extensions:

\begin{observation}\label{O.wRAimpliesExistClosr} The weak resurrection axiom $\wRA(\Gamma)$ implies that $\Hc$ is existentially closed in all forcing extensions by posets from $\Gamma$. That is, $\wRA(\Gamma)$ implies that $\Hc\elesub_{\Sigma_1}V[g]$, whenever $\Q\in \Gamma$ and $g\of\Q$ is $V$-generic.
\end{observation}

\begin{proof}Suppose that $\Q\in\Gamma$ and $g\of\Q$ is $V$-generic. By $\wRA(\Gamma)$, there is $\R\in V[g]$ such that if $h\of \R$ is $V[g]$-generic, then $\Hc\elesub\Hc^{V[g*h]}$. By the \Levy\ absoluteness theorem, which amounts to a simple \Lowenheim-Skolem and reflection argument to collapse the existential witness to a set of hereditary size less than $\continuum$, we have $\Hc^{V[g*h]}\elesub_{\Sigma_1}V[g*h]$ and therefore  $\Hc\elesub\Hc^{V[g*h]}\elesub_{\Sigma_1} V[g*h]$, which implies $\Hc\elesub_{\Sigma_1}V[g*h]$, as desired.
\end{proof}

We shall try in this article to use standard notation. We denote the continuum $2^\omega$ by $\continuum$, and for any infinite cardinal $\delta$, we write $H_\delta$ for the set of all sets hereditarily of size less than $\delta$, that is, with transitive closure of size less than $\delta$. In particular, $H_\continuum$ is the collection of sets hereditarily of size less than the continuum. Relativizing this concept to a particular model of set theory $W$, we write $H_\continuum^W$ to mean the collection of sets in $W$ that are hereditarily of size less than $\continuum^W$ in $W$. Unadorned with such relativizing exponents, notation such as $\continuum$ and $\Hc$ will always refer to the interpretation of these terms in the default ground model $V$. We shall use the notation $f\from X \to Y$ for partial functions, to indicate that $\dom(f)\of X$ and $\ran(f)\of Y$.

\Velickovic\ and Hamkins had initially considered an extreme form of resurrection, the axiom asserting that for every partial order $\Q$, there is $\Rdot$ such that after forcing with $\Q*\Rdot$, there is an elementary embedding $j:V\to V[g*h]$. This axiom, however, is refuted by the generalization of the Kunen inconsistency showing that there is never any nontrivial elementary embedding $j:V\to V[G]$ in any forcing extension $V[G]$ (see \cite{HamkinsKirmayerPerlmutter2012:GeneralizationsOfKunenInconsistency}). Nevertheless, a restriction of the axiom remains interesting: if there is a rank-into-rank embedding $j:V_\lambda\to V_\lambda$, then after certain preparatory forcing $\Vbar=V[G]$, they observed, for any $\Q\in \Vbar_\lambda=V_\lambda[G]\satisfies\ZFC$ there is $\Rdot$, such that in the corresponding extension $\Vbar[g*h]$ there is an elementary embedding $j:\Vbar_\lambda\to \Vbar_\lambda[g*h]$; and one may assume without loss that $\cp(j)=\omega_1$. If one restricts to proper forcing or other classes, then one may insist on $\cp(j)=\omega_2$, and so on. By considering $j\restrict H_\kappa$, where $\kappa=\cp(j)$, one is led directly to the resurrection axioms, which subsequently can be treated, as we do in this article, with a much smaller large cardinal hypothesis.

\section{Resurrection axioms and bounded forcing axioms}\label{S.ConsequencesOfRA}

We regard the resurrection axioms as forcing axioms in light of their consequences amongst the bounded forcing axioms, as in theorem~\ref{T.wRAimpliesBFA_kappa}, and also because they express a precise logical connection between the universe and its forcing extensions. For cardinals $\kappa$ and collections $\Gamma$ of forcing notions, Goldstern and Shelah~\cite{GoldsternShelah1995Nr507:BPFA} introduced the \emph{bounded forcing axiom} $\BFA_\kappa(\Gamma)$, which is the assertion that whenever $\Q\in\Gamma$ and $\B=\text{r.o.}(\Q)$, if $\mathcal{A}$ is a collection of at most $\kappa$ many maximal antichains in $\B\setminus\{0\}$, each antichain of size at most $\kappa$, then there is a filter on $\B$ meeting each antichain in $\mathcal{A}$. With this terminology, $\BFA_\kappa(\ccc)$ is simply the same as Martin's Axiom $\MA(\kappa)$, and having $\BFA_\kappa(\ccc)$ for all $\kappa<\continuum$ amounts to the same as having \MA. The bounded proper forcing axiom $\BPFA$, as defined in~\cite{GoldsternShelah1995Nr507:BPFA}, is the same as $\BFA_{\omega_1}(\proper)$.\footnote{Analogously, the bounded semi-proper forcing axiom \BSPFA\ is the same as $\BFA_{\omega_1}(\semiproper)$, the bounded axiom-A forcing axiom \BAAFA\ is the same as $\BFA_{\omega_1}(\text{axiom-A})$, and the bounded Martin's maximum \BMM\ is the same as $\BFA_{\omega_1}(\Gamma)$ where $\Gamma$ is the class of forcing notions that preserve stationary subsets of $\omega_1$.\label{fn:BFA-Defn}}

\goodbreak
\begin{theorem}\label{T.wRAimpliesBFA_kappa}\
 If $\Gamma$ is any collection of posets, then $\wRA(\Gamma)$ implies $\BFA_\kappa(\Gamma)$ for any $\kappa<\continuum$. In particular,
 \begin{enumerate}
  \item $\wRA(\ccc)$ implies \MA.
 \item $\wRA(\proper)+\neg\CH$ implies \BPFA.
   \item $\wRA(\semiproper)+\neg\CH$ implies $\rm BSPFA$.
 \item $\wRA(\text{axiom-A})+\neg\CH$ implies $\rm BAAFA$.
 \item $\wRA(\text{preserving stationary subsets of }\omega_1)+\neg\CH$ implies $\rm BMM$.
 \end{enumerate}
\end{theorem}

\begin{proof} Assume that $\wRA(\Gamma)$ holds and that $\kappa<\continuum$ is a cardinal.  To verify $\BFA_\kappa(\Gamma)$, fix any $\Q\in \Gamma$ and let $\B=\text{r.o.}(\Q)$ and $\mathcal{A}$ be any collection of $\kappa$ many maximal antichains in $\B\setminus\{0\}$, each antichain of size at most $\kappa$. Let $\B^\prime$ be the subalgebra of $\B$ generated by $\bigcup \mathcal{A}$, so that $\B^\prime\supseteq \bigcup\mathcal{A}$. Then $\B^\prime$ has size at most $\kappa$, and we may assume without loss of generality that both $\mathcal{A}$ and $\B^\prime$ are elements of $H_{\kappa^\plus}$, and thus of $\Hc$, by replacing $\B$ by an isomorphic copy if necessary. If $g\of\B$ is any $V$-generic filter, then it is also $\mathcal{A}$-generic, and so $g\intersect\B^\prime$ meets each antichain in $\mathcal{A}$. Moreover, $g\intersect\B^\prime$ is a filter on $\B^\prime$, since $\B^\prime$ is a subalgebra of $\B$. Thus, there exists in $V[g]$ an $\mathcal{A}$-generic filter on $\B^\prime$. Since $\Hc\elesub_{\Sigma_1}V[g]$ by observation~\ref{O.wRAimpliesExistClosr}, it follows by elementarity that such an $\mathcal{A}$-generic filter on $\B^\prime$ already exists in $V$. This filter generates in $V$ an $\mathcal{A}$-generic filter on $\B$, as desired. Statements (1)-(5) are immediate consequences. \end{proof}

Note that the failure of \CH\ is a necessary hypothesis in statement (2); the resurrection axiom $\RAall$ implies $\wRA(\proper)$, but by theorem~\ref{T.wRANotCollapseBelowC} it also implies \CH, which contradicts \BPFA. For essentially the same reasons, the failure of \CH\ is necessary in statements (3)-(5) also.

As we mentioned, Stavi in the 1980's (see~\cite[thm 25]{StaviVaananen2001:ReflectionPrinciples}) and independently Bagaria~\cite{Bagaria1997:ACharacterizationOfMA} characterized Martin's axiom \MA\ as being equivalent to the assertion that $H_\continuum \elesub_{\Sigma_1} V[g]$ whenever $g\of \Q$ is c.c.c.~forcing. Bagaria~\cite{Bagaria2000:BoundedForcingAxiomsAsGenericAbsoluteness}  generalized this to all bounded forcing axioms $\BFA_\kappa(\Gamma)$, and it follows from his characterization that $\BFA_\kappa(\Gamma)$ is equivalent to $H_{\kappa^\plus}\elesub_{\Sigma_1} V[g]$ whenever $\Q\in\Gamma$ and $g\of\Q$ is $V$-generic, assuming that $\kappa$ is a cardinal of uncountable cofinality and $\Gamma$ is a collection of forcing notions such that $\Q\in\Gamma$ implies $\Q\restrict q\in \Gamma$ for all $q\in\Gamma$. It follows, in particular, that \BPFA\
is equivalent to the assertion that $H_{\omega_2}\elesub_{\Sigma_1} V[g]$ whenever $g\of \Q$ is proper forcing.  Analogous characterizations hold for the axioms \BSPFA, \BAAFA, and \BMM. Moreover, it is easy to see that observation~\ref{O.wRAimpliesExistClosr} and Bagaria's characterization of $\BPFA_\kappa(\Gamma)$ allow for an alternative way of proving theorem~\ref{T.wRAimpliesBFA_kappa}.

\section{Resurrection axioms and the Size of the Continuum} \label{S.RAandContinuum}

Let us now consider the interaction of the resurrection axioms with the size of the continuum.

\begin{theorem}\label{T.wRANotCollapseBelowC}
Under the weak resurrection axiom $\wRA(\Gamma)$, if some forcing $\Q\in\Gamma$ can collapse a cardinal $\delta$, then $\continuum\leq\delta$. Consequently,
\begin{enumerate}
 \item $\RAall$ implies the continuum hypothesis $\CH$.
 \item The weak resurrection axioms for axiom-A forcing, proper forcing, semi-proper forcing, and forcing that preserves stationary subsets of $\omega_1$, respectively, each imply that $\continuum\leq\aleph_2$.
\end{enumerate}
In other words, $\wRA(\Gamma)$ implies that every forcing notion $\Q\in\Gamma$ necessarily preserves all cardinals below $\continuum$.
\end{theorem}

\begin{proof} Assume that $\wRA(\Gamma)$ holds and  $\delta$ is a cardinal below $\continuum$. Suppose for contradiction that $\Q\in \Gamma$ and $g\of\Q$ is $V$-generic such that $\delta$ is collapsed in $V[g]$. Then $\Hc\elesub_{\Sigma_1}V[g]$ by observation~\ref{O.wRAimpliesExistClosr}, and  in $V[g]$, there is a function witnessing that $\delta$ is not a cardinal, but such a function cannot exist in $\Hc$, a contradiction. Statement (1) follows by considering the canonical forcing to collapse $\aleph_1$ and (2) by collapsing $\aleph_2$ using countably closed forcing.
\end{proof}

Justin Moore pointed out that if there are sufficient large cardinals, then the converse of statement (1) is also true. The point is that if projective absoluteness holds, that is, if boldface projective truth is invariant by forcing---and this is a consequence of sufficient large cardinals, such as a proper class of Woodin cardinals---then the theory of $H_{\omega_1}$ with parameters is invariant by forcing, and so $H_{\omega_1}\elesub H_{\omega_1}^{V[g]}$ for any forcing extension. Thus, projective absoluteness implies that $\RAall$ is simply equivalent to \CH, and so we place our focus on the other resurrection axioms. Meanwhile, we do note that $\RAall$ is not equivalent to \CH\ in \ZFC, assuming $\Con(\ZFC)$, because it is equiconsistent with the existence of an uplifting cardinal by theorem~\ref{T.Maintheorem}; see also theorems~\ref{T.RA(Gamma)iffCH} and~\ref{T.RAimplyCUplift}.

A similar argument as in theorem~\ref{T.wRANotCollapseBelowC} shows that under the weak resurrection axiom $\wRA(\Gamma)$ every forcing notion $\Q\in\Gamma$ must necessarily preserve all stationary subsets of ordinals below $\continuum$.
For instance, if $S\of\omega_1$ is any stationary, co-stationary set and $\Q$ is the standard poset  that uses countable conditions to add a club subset $C\of S$, then $\Q$ is countably distributive, but it destroys the stationarity of the complement of $S$. It follows that the weak resurrection axiom $\wRA(\text{countably distributive})$ implies \CH. Moreover, Shelah's~\cite{AbrahamShelah1983:ForcingClosedUnboundedSets} modification of Baumgartner's original poset to add a club $C\of \omega_1$ using finite conditions by restricting it to a stationary set, provides an example of a cofinality-preserving forcing notion that can destroy the stationarity of a subset of $\omega_1$. It follows that the weak resurrection axiom $\wRA(\text{cofinality-preserving})$ implies \CH.

We shall show in section~\ref{S.ExactStrength}, relative to the existence of an uplifting cardinal, that several instances of the resurrection axioms, such as $\RA(\proper)$, $\RA(\text{axiom-A})$, and $\RA(\semiproper)$, are consistent with $\continuum=\aleph_2$, the maximal possible size for the continuum under these axioms by theorem~\ref{T.wRANotCollapseBelowC}. Relative to a supercompact uplifting cardinal, we show in section~\ref{S.RAandPFA} that $\RA(\text{preserving stationary subsets of }\omega_1)$ is consistent with $\continuum=\aleph_2$. Meanwhile, each of these axioms is relatively consistent with $\CH$:

\begin{theorem}\label{T.RA(prop)ConsisWithCH}
The resurrection axiom $\RA(\proper)$ is relatively consistent with \CH. The same is true of the axioms $\RA(\text{axiom-A})$, $\RA(\semiproper)$ and
 $\RA(\text{preserving stationary subsets of }\omega_1)$, and of $\RA(\Gamma)$ for any class $\Gamma$ necessarily closed under finite iterations and containing a poset forcing $\CH$ without adding reals.
\end{theorem}

\begin{proof} Let us illustrate in the case of proper forcing. Suppose that $\RA(\proper)$ holds, and $G\of\P$ is $V$-generic, where $\P=\Add(\omega_1,1)$ is the canonical forcing of the \CH. Consider any proper $\Q\in V[G]$. Since $\P*\Qdot$ is proper in $V$, there is further proper forcing $\Rdot$ such that if $G*g*h\of\P*\Qdot*\Rdot$ is $V$-generic, then $\Hc\elesub \Hc^{V[G*g*h]}$. Restricting this to the countable sets, it follows that $H_{\omega_1}\elesub H_{\omega_1}^{V[G*g*h]}$. Let $\R_2=\Add(\omega_1,1)^{V[G*g*h]}$ be further forcing to recover the \CH\ once again, and suppose $h_2\of\R_2$ is $V[G*g*h]$-generic. Since $\P$ adds no reals over $V$ and $\R_2$ adds no reals over $V[G*g*h]$, we have $H_{\omega_1}=H_{\omega_1}^{V[G]}$ and $H_{\omega_1}^{V[G*g*h]}=H_{\omega_1}^{V[G*g*h*h_2]}$. In other words, we have $H_{\omega_1}^{V[G]}\elesub H_{\omega_1}^{V[G][g*h*h_2]}$. Since $\omega_1=\continuum$ in both $V[G]$ and $V[G*g*h*h_2]$, this witnesses $\RA(\proper)$ in $V[G]$, as desired. An identical argument works with axiom-A forcing, semi-proper forcing, forcing that preserves stationary subsets of $\omega_1$, and with any class $\Gamma$ necessarily closed under finite iterations and containing a poset forcing \CH\ without adding reals.
\end{proof}

In the case of c.c.c.~forcing, we get a dramatic failure of $\CH$:

\begin{theorem}\label{T.RAcccImplyCWeakInacc}
 The resurrection axiom $\RA(\ccc)$ implies that the continuum $\continuum$ is a weakly inaccessible cardinal, even weakly hyper-inaccessible, a limit of such cardinals and so on. In particular, $\RA(\ccc)$ implies that $\CH$ fails spectacularly.
\end{theorem}

\begin{proof}
Assume $\RA(\ccc)$. By theorem~\ref{T.wRAimpliesBFA_kappa}, it follows that \MA\ holds and so $\continuum$ is regular. Let $\Q=\Add(\omega,\continuum^\plus)$ be the forcing to add $\continuum^\plus$ many Cohen reals. By $\RA(\ccc)$, there is further c.c.c.~forcing $\Rdot$ such that if $g*h\of\Q*\Rdot$ is $V$-generic, then $\Hc\elesub \Hc^{V[g*h]}$. Since $V[g*h]$ is a c.c.c.~extension, cardinals are preserved and $\continuum^V$ is a cardinal less than $\continuum^{V[g*h]}$, and therefore an element of $\Hc^{V[g*h]}$. The continuum $\continuum$ cannot be a successor cardinal in $V$, since otherwise $\continuum=\delta^\plus$ for some $\delta<\continuum$ and $\Hc$ would see that $\delta$ is the largest cardinal, but $\Hc^{V[g*h]}$ would not agree. Thus, $\continuum$ is a regular limit cardinal, and hence weakly inaccessible. It must be a limit of such cardinals, that is, weakly $1$-inaccessible, because if the weakly inaccessible cardinals below $\continuum$ were bounded by some $\gamma<\continuum$, then by elementarity, this would also be true in $H_\continuum^{V[g*h]}$, contradicting the fact that $\continuum^V$ remains weakly inaccessible in the c.c.c.~extension $V[g*h]$. Essentially the same argument shows that $\continuum$ is weakly $\alpha$-inaccessible for every $\alpha<\continuum$---so it is weakly hyper-inaccessible---and it is a limit of such cardinals, and so on.
\end{proof}

Although $\RA(\ccc)$ remains compatible with much stronger properties for the continuum $\continuum$, we cannot expect to strengthen the conclusion of theorem \ref{T.RAcccImplyCWeakInacc} to assert, for example, that $\continuum$ is weakly Mahlo, while still assuming only $\ZFC+\RA(\ccc)$ in the hypothesis, since this would imply that it is a Mahlo cardinal in $L$, but this already exceeds the consistency strength of $\RA(\ccc)$ by theorem \ref{T.Maintheorem}, which shows it to be equiconsistent with an uplifting cardinal and therefore strictly weaker than the existence of a Mahlo cardinal.

We pointed out after theorem~\ref{T.wRANotCollapseBelowC} that under projective absoluteness, then also $\RAall$ and \CH\ are equivalent. The next theorem provides instances of resurrection that are outright  equivalent to \CH.

\begin{theorem}\label{T.RA(Gamma)iffCH}
The following are equivalent:
\begin{enumerate}
\item the continuum hypothesis \CH\
 \item $\RA(\text{countably closed})$
 \item $\RA(\text{countably distributive})$
 \item $\RA(\text{does not add reals})$
\item $\wRA(\text{does not add reals})$
\item $\wRA(\text{countably distributive})$
 \end{enumerate}
\end{theorem}

\begin{proof} We first illustrate the equivalence of statements (1) and (2). For the forward direction, suppose that $\CH$ holds, and that $g\of\Q$ is countably closed forcing. Since $\Q$ does not add any new reals, it follows that $\CH$ holds in $V[g]$ and that $H_{\omega_1}=H_{\omega_1}^{V[g]}$. Consequently $\Hc=\Hc^{V[g]}$, and trivial forcing $h\of\R$ over $V[g]$ yields $\Hc\elesub\Hc^{V[g*h]}$, as desired. For the backward direction, assume that $\RA(\text{countably closed})$ holds and $g\of \Q$ is the canonical poset to force $\CH$, using countable conditions. The poset $\Q$ is countably closed and forces $\CH$ in $V[g]$. By $\RA(\text{countably closed})$ there is further countably closed forcing $h\of \R\in V[g]$ such that $\Hc\elesub\Hc^{V[g*h]}$. Since $\R$ does not add any reals, it follows that $\CH$ holds in $V[g*h]$, and consequently by elementarity $\CH$ also holds in $V$, as desired.

Essentially the same argument establishes the equivalence of $\CH$ with (3), and also with (4). Lastly, note that (4) implies (5), which in turn implies (6), and we saw earlier in the remarks after theorem~\ref{T.wRANotCollapseBelowC} that statement (6) implies (1).
\end{proof}

Suppose that $\delta\geq\aleph_1$ is a regular cardinal, and $\Gamma$ is a class of forcing notions necessarily containing a poset which forces $\continuum\leq\delta$ such that posets in $\Gamma$ do not add bounded subsets of $\delta$. Then similar arguments as used in theorem~\ref{T.RA(Gamma)iffCH} show that $\continuum\leq\delta$ is equivalent to $\RA(\Gamma)$, and they also show that each of the resurrection axioms $\RA(\ltdelta\text{-closed})$ and $\RA(\ltdelta\text{-distributive})$ is equivalent to $\continuum\leq\delta$. 
 
We conclude this section by pointing out that some natural-seeming resurrection principles are simply inconsistent.

\begin{theorem}\label{T.InconsistentRA}\
 \begin{enumerate}
 \item $\RA(\delta\text{-c.c.})$ is inconsistent, for any cardinal $\delta\geq\aleph_2$.
 \item $\RA(\text{cardinal-preserving})$ is inconsistent.
 \item $\RA(\text{cofinality-preserving})$ is inconsistent.
 \item $\RA(\aleph_1\text{-preserving}\intersect \aleph_2\text{-preserving})$ is inconsistent.
\end{enumerate}

\end{theorem}

\begin{proof}
For (1), fix any cardinal $\delta\geq\aleph_2$ and assume $\RA(\delta\text{-c.c})$. Since the usual forcing to collapse $\aleph_1$ is $\aleph_2\text{-c.c.}$, and therefore $\delta\text{-c.c.}$, it follows by theorem~\ref{T.wRANotCollapseBelowC} that $\CH$ holds in $V$.  If we force to add $g\of \Q=\text{Add}(\omega,\delta^\plus)$, then all cardinals are preserved and $\continuum=\delta^\plus$ in $V[g]$. By $\RA(\delta\text{-c.c.})$ there is further $\delta\text{-c.c.}$ forcing $h\of \R$ in $V[g]$ such that $\Hc\elesub \Hc^{V[g*h]}$. Since $\R$ preserves the cardinals $\delta$ and $\delta^\plus$ as two distinct uncountable cardinals, it follows that $\continuum\geq\delta^\plus$ and consequently that $\CH$ fails in $V[g*h]$, a contradiction to the elementarity $\Hc\elesub \Hc^{V[g*h]}$.

For (2), assume $\RA(\text{cardinal-preserving})$. The weak resurrection axiom $\wRA(\text{cofinality-preserving})$ holds, and so $\CH$ holds by our remarks after theorem~\ref{T.wRANotCollapseBelowC}. Moreover, if we force with $g\of \Q=\text{Add}(\omega,\aleph_2)$, then $\continuum=\aleph_2$ in $V[g]$ and the same argument as in theorem~\ref{T.RAcccImplyCWeakInacc}, but now for any cardinal-preserving forcing $h\of \R\in V[g]$ rather than c.c.c.~forcing, shows that $\continuum\geq \aleph_2$ in $V[g*h]$ and thus $\Hc\not \elesub \Hc^{V[g*h]}$, a contradiction.

Statement (3) is proved by the same argument as for (2), but now for cofinality-preserving forcing notions, and the argument for (4) is similar also, since again $\CH$ holds in $V$, and if $g\of \Q=\text{Add}(\omega,\aleph_2)$ is $V$-generic, then it suffices to know that $\R\in V[g]$ preserves the cardinals $\aleph_1$ and $\aleph_2$ to conclude that $\continuum \geq \aleph_2$ in $V[g*h]$ and therefore that $\Hc\not \elesub \Hc^{V[g*h]}$, a contradiction.
 \end{proof}

\section{The uplifting cardinals}\label{S.UpliftingCardinals}

In this section, we introduce the uplifting cardinals. We view the uplifting cardinals as relatively low in the large cardinal hierarchy, in light of the bounds provided by theorem~\ref{T.ConBoundUplift}. Uplifting cardinals relativize to $L$, and they have what we call a \HOD-anticipating uplifting Laver function, as in statement~(2) of theorem~\ref{T.VariousUpliftingLaverFcns}. In section~\ref{S.ExactStrength}, we shall show that many instances of resurrection axioms are equiconsistent with the existence of an uplifting cardinal.

\begin{definition}\label{D.ThetaUplift}\rm
 An inaccessible cardinal $\kappa$ is \emph{uplifting} if for every ordinal $\theta$ it is $\theta$-uplifting, meaning that there is an inaccessible $\gamma\geq\theta$ such that $V_\kappa\elesub V_\gamma$ is a proper elementary extension. An inaccessible cardinal is \emph{pseudo uplifting} if for every ordinal $\theta$ it is pseudo $\theta$-uplifting, meaning that there is a cardinal $\gamma\geq\theta$ such that $V_\kappa\elesub V_\gamma$ is a proper elementary extension, without insisting that $\gamma$ is inaccessible.
\end{definition}

It is an elementary exercise to see that if $V_\kappa\elesub V_\gamma$ is a proper elementary extension, then $\kappa$ and hence also $\gamma$ are $\beth$-fixed points, and so $V_\kappa=H_\kappa$ and $V_\gamma=H_\gamma$. It follows that a cardinal $\kappa$ is uplifting if and only if it is regular and there are arbitrarily large regular cardinals $\gamma$ such that $H_\kappa\elesub H_\gamma$. It is also easy to see that every uplifting cardinal $\kappa$ is uplifting in $L$, with the same targets. Namely, if $V_\kappa\elesub V_\gamma$, then we may simply restrict to the constructible sets to obtain $V_\kappa^L=L^{V_\kappa}\elesub L^{V_\gamma}=V_\gamma^L$. An analogous result holds for pseudo uplifting cardinals.

The {\df \Levy\ scheme} is the theory ``$V_\delta\!\elesub \!V \,+ \delta$ is inaccessible,'' which is formalized in the language of set theory augmented with a constant symbol for $\delta$, consisting of the axioms \hbox{$\forall x\in V_\delta\, [\varphi(x)\iff\varphi^{V_\delta}(x)]$,} plus the assertion that $\delta$ is inaccessible. The \Levy\ scheme has figured in various other consistency results, such as the boldface maximality principle $\MP(\R)$, as in \cite{Hamkins2003:MaximalityPrinciple} or \cite{StaviVaananen2001:ReflectionPrinciples}. The \Levy\ scheme implies the theory ``$\Ord$ is Mahlo'', the scheme asserting of every definable closed unbounded class of ordinals that it contains a regular cardinal, and a simple compactness argument shows that these two theories are equiconsistent. The consistency strength of the existence of an uplifting cardinal is bounded above and below by:

\begin{theorem}\label{T.ConBoundUplift}\
\begin{enumerate}
 \item If $\delta$ is a Mahlo cardinal, then $V_\delta$ has a proper class of uplifting cardinals.
 \item Every uplifting cardinal is pseudo uplifting and a limit of pseudo uplifting cardinals.
 \item If there is a pseudo uplifting cardinal, or indeed, merely a pseudo $0$-uplifting cardinal, then there is a transitive set model of $\ZFC+{}$the \Levy\ scheme, and consequently a transitive model of $\ZFC+\Ord$ is Mahlo.
\end{enumerate}
\end{theorem}

\begin{proof} For (1), suppose that $\delta$ is a Mahlo cardinal. By the \Lowenheim-Skolem theorem, there is a club set $C\of\delta$ of cardinals $\beta$ with $V_\beta\elesub V_\delta$. Since $\delta$ is Mahlo, the club $C$ contains unboundedly many inaccessible cardinals. If $\kappa<\gamma$ are both in $C$, then $V_\kappa\elesub V_\gamma$, as desired. Similarly, for (2), if $\kappa$ is uplifting, then $\kappa$ is pseudo uplifting and if $V_\kappa\elesub V_\gamma$ with $\gamma$ inaccessible, then there are unboundedly many ordinals $\beta<\gamma$ with $V_\beta\elesub V_\gamma$ and hence $V_\kappa\elesub V_\beta$. So $\kappa$ is pseudo uplifting in $V_\gamma$, and it follows that there must be unboundedly many pseudo uplifting cardinals below $\kappa$.  For (3), if $\kappa$ is inaccessible and $V_\kappa\elesub V_\gamma$, then $V_\gamma$ is a transitive set model of $\ZFC+$the \Levy\ scheme, and thus also a model of the scheme ``$\Ord$ is Mahlo.''
\end{proof}

So the existence of an uplifting cardinal, if consistent, is in consistency strength strictly between the existence of a Mahlo cardinal and the scheme ``$\Ord$ is Mahlo.'' We take these bounds both to be rather close together and also to be rather low in the large cardinal hierarchy. Note that a pseudo 0-uplifting cardinal is the same thing as a $0$-extendible cardinal. As a refinement of the \Levy\ scheme, recall that for any given natural number $n$, an inaccessible cardinal $\kappa$ is {\df $\Sigma_n$-reflecting} if $H_\kappa\elesub_{\Sigma_n} V$. Recall also that $H_\kappa\elesub_{\Sigma_1}V$ whenever $\kappa$ is any uncountable cardinal.

\begin{observation}\label{O.UpliftSigma3Reflectg}\
\begin{enumerate}
     \item Every uplifting cardinal is a limit of $\Sigma_3$-reflecting cardinals, and is itself $\Sigma_3$-reflecting.
     \item If $\kappa$ is the least uplifting cardinal, then $\kappa$ is not $\Sigma_4$-reflecting, and there are no $\Sigma_4$-reflecting cardinals below $\kappa$.
 \end{enumerate}
\end{observation}

\begin{proof} For (1), suppose that $\kappa$ is uplifting, and let us first show that $\kappa$ is $\Sigma_3$-reflecting. Thus, assume that $V\models\exists x \varphi(x,a)$ for some $\Pi_2$ formula $\varphi(x,y)$ and some $a\in V_\kappa$. Let $x_0$ be a witness such that $V\models\varphi(x_0,a)$, and let $\gamma$ be any uncountable cardinal with $x_0\in V_\gamma$ such that $V_\kappa\elesub V_\gamma$. Since $V_\gamma$ is existentially closed in $V$, it follows that $\Pi_2$ formulas are downwards absolute to $V_\gamma$, and so $V_\gamma\models\varphi(x_0,a)$, which implies by elementarity that $V_\kappa\models\exists x\varphi(x,a)$, as desired; the converse direction is easier. Next, suppose for contradiction that the set of $\Sigma_3$-reflecting cardinals is bounded below $\kappa$. Then $V_\kappa$ sees this bound, since $\kappa$ is $\Sigma_3$-reflecting. Thus, if $V_\kappa\elesub V_\gamma$, then $V_\gamma$ thinks that the set of $\Sigma_3$-reflecting cardinals is bounded below $\kappa$. But this is impossible, since $\kappa$ itself is (much more than) $\Sigma_3$-reflecting in $V_\gamma$.

Statement (2) is an immediate consequence of the fact that the property of being uplifting is $\Pi_3$ expressible, and so the existence of an uplifting cardinal is a $\Sigma_4$ assertion.
\end{proof}

The analogous observation for pseudo uplifting cardinals holds as well, namely, every pseudo uplifting cardinal is $\Sigma_3$-reflecting and a limit of $\Sigma_3$-reflecting cardinals; and if $\kappa$ is the least pseudo uplifting cardinal, then $\kappa$ is not $\Sigma_4$-reflecting, and there are no $\Sigma_4$-reflecting cardinals below $\kappa$.

For an uplifting cardinal $\kappa$, we say that a function $f\from \kappa \to \kappa$ has the {\df uplifting Menas property} for $\kappa$ if for every ordinal $\theta$ there is an inaccessible cardinal $\gamma$ above $\theta$ and a function $f^*\from\gamma\to\gamma$ such that $\< V_\kappa,f> \elesub \< V_\gamma,f^*>$ and $\theta\leq f^*(\kappa)$.\footnote{Analogous Menas  properties of functions $f\from\kappa\to\kappa$  are defined for various large cardinals $\kappa$, not just for uplifting cardinals (see~\cite{Hamkins2000:LotteryPreparation}), and their definitions change depending on the particular large cardinal in question. However, we will simply refer to it as the \emph{Menas property} for $\kappa$ if it is clear from context which large cardinal property of $\kappa$ we are concerned with.} In the cases below where the function $f$ is actually definable in $V_\kappa$, then of course we needn't add it as a separate predicate to the structure, and it will suffice that $V_\kappa\elesub V_\gamma$ and $\theta\leq f^*(\kappa)$, where $f^*$ is the corresponding function defined in $V_\gamma$.

\begin{theorem}\label{T.UpliftImplyMenas} Every uplifting cardinal has a function with the Menas property. Indeed, there is a class function $f\from\ORD\to\ORD$ such that for every uplifting cardinal $\kappa$, the restriction $f\restrict\kappa\from\kappa\to\kappa$ has the Menas property for $\kappa$, and $f\restrict\kappa$ is a definable class in $V_\kappa$.
\end{theorem}

\begin{proof}
The \emph{failure-of-upliftingness} function $f\from\Ord\to\Ord$ has the desired property. Namely, if $\delta$ is a cardinal but not uplifting, then let $f(\delta)$ be the supremum of the inaccessible cardinals $\gamma$ for which $V_\kappa\elesub V_\gamma$. If $\kappa$ is uplifting, then by the elementarity of $V_\kappa\elesub V_\gamma$ for increasingly large $\gamma$, it follows that $V_\kappa$ correctly computes the value of $f(\delta)$ for every $\delta<\kappa$. In particular, $f(\delta)<\kappa$ for any non-uplifting cardinal $\delta<\kappa$, so that   $f\image\kappa\of\kappa$, and the restriction $f\restrict\kappa$ is the failure-of-upliftingness function as defined in $V_\kappa$.

To see that $f\restrict\kappa$ has the Menas property for $\kappa$, fix any ordinal $\theta$ and any inaccessible cardinal $\gamma\geq\theta$ with $V_\kappa\elesub V_\gamma$. Applying the fact that $\kappa$ is uplifting again, let $\lambda$ be the smallest inaccessible cardinal above $\gamma$ for which $V_\kappa\elesub V_\lambda$. If $f^*\from\lambda\to\lambda$ is the failure-of-upliftingness function as defined in $V_\lambda$, then since this function is definable, we have $\<V_\kappa,f>\elesub \<V_\lambda,f^*>$. And because $V_\lambda$ can see that $V_\kappa\elesub V_\gamma$, but by the minimality of $\lambda$ can see no higher inaccessible cardinal to which $V_\kappa$ extends elementarily, it follows that $f^*(\kappa)=\gamma$, thereby witnessing the desired Menas property.
\end{proof}

We now strengthen the previous theorem by proving that every uplifting cardinal has functions with certain uplifting Laver properties, properties that strengthen the uplifting Menas property of theorem~\ref{T.UpliftImplyMenas} significantly. We shall see in section~\ref{S.ExactStrength} that the uplifting Menas property suffices to obtain equiconsistency results for instances of resurrection such as $\RAall$, $\RA(\proper)+\neg\CH$, $\RA(\semiproper)+\neg\CH$ and others, but it does not seem to suffice to obtain the corresponding result for $\RA(\ccc)$.

If $\kappa$ is an uplifting cardinal, define that \hbox{$\ell\from\kappa\to H_\kappa$} is an \emph{uplifting Laver function} for $\kappa$, if for every set $x$ there are unboundedly many inaccessible cardinals $\gamma$ with a corresponding function $\ell^*\from\gamma\to H_\gamma$ such that $\<H_\kappa,\ell>\elesub\<H_\gamma,\ell^*>$ and $\ell^*(\kappa)=x$. Following the scheme of axioms in \cite{Hamkins:LaverDiamond}, let us say that the uplifting Laver Diamond $\LDuplift_\kappa$ holds at $\kappa$ when there is such a function $\ell\from\kappa\to H_\kappa$. For a natural weakening of this concept, we say that $\ell \from \kappa \to \kappa$ is an {\df ordinal-anticipating} uplifting Laver function for $\kappa$, if for every ordinal $\beta$ there are unboundedly many inaccessible cardinals $\gamma$ with a corresponding function $\ell^*\from\gamma\to\gamma$ such that $\< H_\kappa,\ell> \elesub \< H_\gamma,\ell^*>$ and $\ell^*(\kappa)=\beta$. Similarly, we have the concept of a {\df $\HOD$-anticipating} uplifting Laver function, where we can achieve $\ell^*(\kappa)=x$ for any $x\in\HOD$.

We think of the uplifting Laver functions as in statement~(3) of the following theorem as the world's smallest Laver functions, in light of the fact that uplifting is weaker than Mahlo.

\begin{theorem}\label{T.VariousUpliftingLaverFcns}\

\begin{enumerate}
\item Every uplifting cardinal $\kappa$ has an ordinal-anticipating uplifting Laver function $\ell\from\kappa \to \kappa$ definable in $H_\kappa$.
\item Every uplifting cardinal $\kappa$ has a $\HOD$-anticipating uplifting Laver function $\ell\from\kappa\to H_\kappa$ definable in $H_\kappa$.
\item If $V=\HOD$, then every uplifting cardinal $\kappa$ has an uplifting Laver function $\ell\from\kappa\to H_\kappa$ definable in $H_\kappa$.
  \end{enumerate}
\end{theorem}

\begin{proof}
For (1), working in $H_\kappa$, define that $\ell(\delta)=\beta$, if $\delta$ is a cardinal and the collection of inaccessible cardinals $\xi$ above $\delta$ with $H_\delta\elesub H_\xi$ has order type exactly $\theta+\beta$ for some infinite cardinal $\theta$ for which $\beta<\theta$. (Note that the decomposition $\theta+\beta$ is unique.) To see that $\ell$ is an ordinal-anticipating uplifting Laver function, fix any ordinal $\beta$ and any infinite cardinal $\theta$ above $\beta$. Since $\kappa$ is uplifting, there are unboundedly many inaccessible cardinals $\xi>\kappa$ with $H_\kappa \elesub H_\xi$. Let $\gamma$ be the $(\theta+\beta)^\th$ such $\xi$. In this case, there are precisely $\theta+\beta$ many such $\xi$ inside $H_\gamma$ with $H_\kappa\elesub H_\xi$, and so $\ell^*(\kappa)=\beta$, where $\ell^*$ is defined in $H_\gamma$ just as $\ell$ is defined in $H_\kappa$. The elementarity $H_\kappa\elesub H_\gamma$ extends to $\< H_\kappa,\ell> \elesub \< H_\lambda,\ell^*>$, because $\ell$ is definable, witnessing the desired instance of the Laver property.

For (2), let $f\from\kappa\to\kappa$ be an ordinal-anticipating uplifting Laver function, definable in $H_\kappa$, as in statement~(1). Define that $\ell(\delta)=x$, if $f(\delta)=\<\theta,\beta>$ is the ordinal code for a pair of ordinals, such that $x$ is ordinal definable in $V_\theta$ and $x$ is the $\beta^\th$ element of $\HOD^{V_\theta}$ using the definable well-ordering of $\HOD$ inside $V_\theta$. To see that $\ell$ is a $\HOD$-anticipating uplifting Laver function, suppose that $x\in\HOD$. It follows by reflection that $x$ is in the $\HOD$ of some $V_\theta$ and is the $\beta^\th$ element of $\HOD^{V_\theta}$ for some $\beta$. Since $f$ is an ordinal-anticipating uplifting Laver function, there are arbitrarily large inaccessible cardinals $\gamma$ for which $V_\kappa\elesub V_\gamma$ and $f^*(\kappa)=\<\theta,\beta>$, where $f^*$ is defined in $V_\gamma$ by the same definition of $f$ in $V_\kappa$. By construction, we have $\ell^*(\kappa)=x$, where $\ell^*$ is defined in $V_\gamma$ in analogy with $\ell$ in $V_\kappa$, witnessing the desired instance of the Laver property.

Statement~(3) is immediate from~(2).
\end{proof}

Just as with theorem~\ref{T.UpliftImplyMenas}, the proof of~(1) shows that there is a global ordinal-anticipating uplifting Laver function, a class function $\ell\from\ORD\to\ORD$ such that for every uplifting cardinal $\kappa$ the restriction $\ell\restrict\kappa\from\kappa\to\kappa$ is an ordinal-anticipating uplifting Laver function for $\kappa$, and $\ell\restrict\kappa$ is definable in $H_\kappa$. The proof of~(2) shows that there is a global $\HOD$-anticipating uplifting Laver function, defined accordingly. Statement~(3) asserts that $V=\HOD$ implies $\LDuplift_\kappa$ for every uplifting cardinal $\kappa$. Following~\cite{Hamkins:LaverDiamond}, we define $\LDuplift$ to be the assertion that there is a global uplifting Laver function, a class function $\ell\from\ORD\to V$ such that for every uplifting cardinal $\kappa$ the restriction $\ell\restrict\kappa\from\kappa\to H_{\kappa}$ is an uplifting Laver function for $\kappa$, and $\ell\restrict\kappa$ is definable in $H_\kappa$. We have thus proved that $V=\HOD$ implies $\LDuplift$.

\begin{question}
 Can there be an uplifting cardinal with no uplifting Laver function? In other words, is it consistent that $\kappa$ is uplifting ${}+\neg\LDuplift_\kappa$?
\end{question}

Although we have proved that an uplifting cardinal can have a Laver function, we would like to remark that there is no analogue here of the Laver preparation that makes an uplifting cardinal Laver indestructible, because the main result of \cite{BagariaHamkinsTsaprounisUsuba:SuperstrongAndOtherLargeCardinalsAreNeverLaverIndestructible} shows that uplifting cardinals and even pseudo uplifting cardinals are never Laver indestructible.

\section{The exact large cardinal strength of the resurrection axioms}\label{S.ExactStrength}

In this section, we prove that many instances of the resurrection axioms, including $\RAall$, $\RA(\ccc)$, $\RA(\proper)+\neg\CH$ and others, are each equiconsistent with the existence of an uplifting cardinal.
The proof outline proceeds in two directions: on the one hand, theorem \ref{T.RAimplyCUplift} shows that many instances of the (weak) resurrection axioms imply that $\continuum^V$ is uplifting in $L$; and conversely, given any uplifting cardinal $\kappa$, we may perform a suitable lottery iteration of $\Gamma$ forcing to obtain the resurrection axiom for $\Gamma$ in a forcing extension with $\kappa=\continuum$. The main result is stated in theorem~\ref{T.Maintheorem}.

\begin{theorem}\label{T.RAimplyCUplift}\
 \begin{enumerate}
  \item $\RAall$ implies that $\continuum^V$ is uplifting in $L$.
   \item $\RA(\ccc)$ implies that $\continuum^V$ is uplifting in $L$.
   \item $\wRA(\text{countably closed})+\neg\CH$ implies that $\continuum^V$ is uplifting in $L$.
   \item  Under $\neg\CH$, the weak resurrection axioms for the classes of axiom-A forcing, proper forcing, semi-proper forcing, and posets that preserve stationary subsets of $\omega_1$, respectively, each imply  that $\continuum^V$ is uplifting in $L$.
 \end{enumerate}
\end{theorem}

\begin{proof}
For (1), suppose that $\RAall$ holds. Then $\CH$ holds by theorem~\ref{T.wRANotCollapseBelowC}. Let $\kappa=\continuum=\omega_1$. To see that $\kappa$  is uplifting in $L$, it suffices by the remarks after definition~\ref{D.ThetaUplift} to show that $\kappa$ is regular in $L$, and that $H_\kappa^L\elesub H_\gamma^L$ for arbitrarily large ordinals $\gamma$ that are regular cardinals in $L$. The cardinal $\kappa$ is regular and therefore regular in $L$. Thus, fix any cardinal $\alpha >\kappa$, and let $\Q$ be a poset that collapses $\alpha$ to $\aleph_0$. By \RAall, there
is further forcing $\Rdot$, such that if $g*h\of\Q*\Rdot$ is $V$-generic, then $\Hc\elesub \Hc^{V[g*h]}$. Let $\gamma=\continuum^{V[g*h]}$. Since $\alpha$ was made countable in $V[g]$, it follows that $\alpha\lt\gamma$. Since $\Hc$ believes that every ordinal is countable, this is also true by elementarity in $\Hc^{V[g*h]]}$, and so $\CH$ holds in $V[g*h]$. It follows that $\gamma=\omega_1^{V[g*h]}$, and so $\gamma$ is regular in $V[g*h]$ and therefore also in $L$, with $\kappa < \alpha<\gamma$ and $H_\kappa^V \elesub H_\gamma^{V[g*h]}$. By relativizing formulas to the constructible sets, it follows that $H_\kappa^L=(H_\kappa^V\intersect L)\elesub (H_\gamma^{V[g*h]}\intersect L)=H_\gamma^L$, as desired.

For (2), suppose that $\RA(\ccc)$ holds, and let $\kappa=\continuum$. By theorem~\ref{T.RAcccImplyCWeakInacc}, we know that $\kappa$ is weakly inaccessible and therefore  inaccessible in $L$. Again, fix any cardinal $\alpha>\kappa$ and let $\Q=\text{Add}(\omega,\alpha)$ be now the forcing that adds $\alpha$ many Cohen reals. By $\RA(\ccc)$ there is further c.c.c.~forcing $\Rdot$ such that if $g*h\of\Q*\Rdot$ is $V$-generic, then $\Hc\elesub\Hc^{V[g*h]}$. Since $\continuum^{V[g]}\geq \alpha>\kappa$ and $\R$ does not collapse cardinals, it follows that $\continuum^{V[g*h]}\geq \alpha >\kappa$.  Since $\MA$ holds in $\Hc$, it holds in $\Hc^{V[g*h]}$ by elementarity, and hence also in $V[g*h]$. It follows that  $\continuum^{V[g*h]}$ is regular in $V[g*h]$ and hence in $L$. By relativizing formulas to $L$ it follows again that $H_\kappa^L\elesub H_\gamma^L$, where $\gamma=\continuum^{V[g*h]}$, as desired.

For (3), suppose that $\wRA(\text{countably closed})$ holds and $\CH$ fails. Then $\continuum = \aleph_2$ by theorem~\ref{T.wRANotCollapseBelowC}. Let $\kappa=\continuum=\aleph_2$, which is regular and therefore regular in $L$. Again, fix any cardinal $\alpha>\kappa$ and let $\Q$ be the countably closed forcing that collapses $\alpha$ to $\aleph_1$ using countable conditions. By $\wRA(\text{countably closed})$ there is $\Rdot$, such that if $g*h\of\Q*\Rdot$ is $V$-generic, then $\Hc\elesub \Hc^{V[g*h]}$. Since $\aleph_1^V$ is the largest cardinal in the former structure, this must also be true in $\Hc^{V[g*h]}$, and so $\continuum=\aleph_2^{V[g*h]}$. As $\alpha$ is an ordinal of size $\aleph_1$ in $V[g]$, and thus in $V[g*h]$, it follows that $\continuum^{V[g*h]}>\alpha$. If we let $\gamma=\continuum^{V[g*h]}=\aleph_2^{V[g*h]}$, then $\gamma$ is a regular cardinal above $\alpha$ in $V[g*h]$ and hence in $L$, and by relativizing formulas to $L$ it follows again that $H_\kappa^L\elesub H_\gamma^L$, as desired. Statement (4) is an immediate consequence of (3), since countably closed forcing is included in all those other classes of forcing.
\end{proof}

The failure of $\CH$ is a necessary assumption in statement (3) of theorem~\ref{T.RAimplyCUplift}, because $\wRA(\text{countably closed})$ holds in $L$ by theorem~\ref{T.RA(Gamma)iffCH}, but $\continuum^L$ is of course not uplifting in $L$. And it is similarly required in statement (4), since in the case of proper forcing, for instance, if $\kappa$ is an uplifting cardinal in $L$, we shall see later by first applying theorem~\ref{T.UpliftImplyRA(prop)} and then theorem~\ref{T.RA(prop)ConsisWithCH} that there is a proper forcing extension $L[G]$ of $L$ satisfying  $\RA(\proper)+\CH$, but  $\continuum^{L[G]}$ is not uplifting in $L$ as $\continuum^{L[G]}=\aleph_1^{L[G]}=\aleph_1^L$.

We now turn to the converse consistency implications, producing models of $\RA(\Gamma)$ for various natural forcing classes $\Gamma$ from models with an uplifting cardinal. In order to produce models of $\RA(\Gamma)$, our main tool will be to undertake various instances of what we call a lottery iteration, a forcing iteration in which each stage of forcing performs a lottery sum. The idea goes back to the lottery preparation of Hamkins \cite{Hamkins2000:LotteryPreparation}, which was introduced as an alternative to the Laver preparation in order to make large cardinals indestructible in situations where there is no Laver function.

Specifically, if $\mathcal{A}$ is a collection of partial orders, the \emph{lottery sum} of $\mathcal{A}$, denoted $\oplus\mathcal{A}$, is the partial order $\{\< \Q,q> \;|\;q\in\Q\in\mathcal{A}\}\cup\{\one\}$, ordered with $\one$ above everything and $\<\Q,q>\leq\<\P,p>$ if and only if $\Q=\P$ and $q\leq_{\Q}p$. Forcing with $\oplus\mathcal{A}$ amounts to choosing a winning poset from $\mathcal{A}$ and then forcing with it. (See \cite{Hamkins2000:LotteryPreparation}; the lottery sum is also commonly known as side-by-side forcing, and it is forcing equivalent to the Boolean product, without omitting $0$, of the corresponding Boolean algebras.) A {\df lottery iteration} is any forcing iteration in which each stage of forcing is the lottery sum of a collection of forcing notions. For any definable class $\Gamma$ of forcing notions and any $f\from\kappa\to\kappa$, the lottery iteration of $\Gamma$ forcing, relative to $f$, is the iteration of length $\kappa$ (with some specified support) that forces at stage $\beta\in\dom(f)$ with the lottery sum of all posets $\Q$ in $\Gamma^{V[G_\beta]}$ having hereditary size at most $f(\beta)$, and trivial forcing at stages $\beta\notin\dom(f)$. More generally, when we have a notion of what it means to be {\df allowed} at stage $\beta$, that is, if we have definable classes $\Gamma_\beta$ of forcing notions, then the corresponding lottery iteration forces at stage $\beta\in\dom(f)$ with the lottery sum of all $\Q\in\Gamma_\beta^{V[G_\beta]}$ of hereditary size at most $f(\beta)$, and again trivial forcing at stages $\beta\notin\dom(f)$. Although one could incorporate the size restriction imposed by $f(\beta)$ into the definition of $\Gamma_\beta$, it is more convenient to consider these as separate restrictions, one restriction on the type of forcing and another on the size of the forcing.

Just as the lottery preparation of a cardinal $\kappa$ relative to a function $f\from\kappa\to\kappa$ works best when $f$ exhibits a certain fast-growth behavior, called the Menas property in~\cite{Hamkins2000:LotteryPreparation}, the same is true of the lottery iterations considered here; and we proved in theorem~\ref{T.UpliftImplyMenas} that every uplifting cardinal has a function with the Menas property.

The lottery preparation of \cite{Hamkins2000:LotteryPreparation}, for example, is the Easton-support lottery iteration of length $\kappa$, relative to a function $f\from\kappa\to\kappa$ with the Menas property, where posets are allowed at inaccessible stage $\beta$ exactly if they are strategically $\delta$-closed for every $\delta<\beta$. In this article, we shall similarly perform the countable-support lottery iteration of proper forcing, relative to a Menas function $f\from\kappa\to\kappa$, as well as a similar lottery iteration of axiom-A forcing and a revised-countable-support lottery iteration of semi-proper forcing, among others.

A countable-support lottery iteration of proper posets was first employed in the second author's dissertation~\cite{Johnstone2007:Dissertation}, where he used it to prove the relative consistency of a certain fragment of $\PFA$ from a weaker-than-expected large cardinal hypothesis. Similar lottery iterations appear also in \cite{HamkinsJohnstone2009:PFA(aleph_2-preserving)}, including the revised-countable-support lottery iteration of semi-proper posets, which appears independently in~\cite{NeemanSchimmerling2008:HierarchiesOfForcingAxioms}.

Given an uplifting cardinal $\kappa$ and a corresponding elementary extension $H_\kappa\elesub H_\gamma$ for some inaccessible $\gamma>\kappa$, we shall use the next lemma to force with some poset $G\of \P\of H_\kappa$ and lift the elementarity to $H_\kappa[G]\elesub H_\gamma[G^*]$, which will turn out in our case to be the same as $\Hc^{V[G]}\elesub \Hc^{V[G][g*h]}$, which will thereby witness an instance of the resurrection axiom. Since $\P$ will be class forcing from the point of view of $H_\kappa$, rather than set forcing, there will be some complications that we must analyze.

If $M\models\ZFC $ and $A\of M$, then we shall say that the expanded structure $\<M,\in,A>$ satisfies $\ZFC$ to mean that it satisfies the version of \ZFC\ in which we allow a predicate symbols for the class $A$ to be used in instances of the replacement and separation axioms; this theory is also sometimes denoted $\ZFC(A)$. Define that a forcing notion $\P\of M$ is {\df nice} for class forcing over $\mathcal{M}=\<M,\in,A>$, if $\P$ is definable in $\mathcal{M}$, the corresponding forcing relations are definable in $\mathcal{M}$ and the truth lemma---asserting that a statement is true in a forcing extension exactly if it is forced by a condition in the generic filter---holds for forcing with $\P$ over $\mathcal{M}$.

Suppose that $\mathcal{M}\elesub \mathcal{M}^*$ for some model $M^*=\<M^*,\in,A^*>$ with $M^*$ transitive and $A^*\of M^*$, and suppose also that $\P^*\of M^*$ is the analogously defined class in $\mathcal{M}^*$. If $\P$ is nice for class forcing over $\mathcal{M}$, then we say that the niceness of $\P$ is \emph{preserved} to $\mathcal{M}^*$ if $\P^*$ is nice for class forcing over $\mathcal{M}^*$ and the forcing relations for forcing with $\P^*$ over $\mathcal{M}^*$ are defined in $\mathcal{M}^*$ by the same formulas and same parameters as those for forcing with $\P$ over $\mathcal{M}$ are defined in $\mathcal{M}$.

\begin{lemma}[Lifting Lemma] \label{L.LiftgStrUplift}
Suppose that $\<M,\in,A>\elesub\<M^*,\in,A^*>$ are transitive models of $\ZFC$, that $\P$ is a definable class in $\<M,\in,A>$ that is nice for forcing and that the niceness of $\P$ is preserved to the analogous class $\P^*$ defined in $\<M^*,\in,A^*>$. If $G\of\P$ is an $M$-generic filter and $G^*\of\P$ is $M^*$-generic with $G=G^*\intersect\P$, then \hbox{$\<M[G],\in,A,G>\elesub\<M^*[G^*],\in,A^*,G^*>$.}
\end{lemma}

\begin{proof}
Suppose that $\mathcal{M}=\<M,\in,A>$ and $\mathcal{M}^*=\<M^*,\in,A^*>$ are as in the statement of the lemma, with the partial orders $\P$ and $\P^*$ as stated there, with generic filters $G\of\P$ and $G^*\of\P^*$ as supposed. If $\tau$ is any $\P$-name in $M$, then $\tau$ is also a $\P^*$-name in $M^*$, and a simple ${\in}$-induction shows that $\tau_G=\tau_{G^*}$. To see that  \hbox{$\<M[G],\in,A,G>\elesub\<M^*[G^*],\in,A^*,G^*>$},  suppose that $\<M[G],\in,A,G>\models \varphi[\tau_G]$ for some \hbox{$\P$-name $\tau$} and some formula $\varphi$ in the extended language of set theory with two unary predicate symbols. It suffices to show that $\<M^*[G^*],\in,A^*,G^*>\models \varphi[\tau_G]$. Since $\P$ is nice for class forcing over $\mathcal{M}$, there exists some condition $p\in G$ such that $p\forces_\P\varphi(\tau)$, and this statement is definable in $\mathcal{M}$. Since the niceness of $\P$ is preserved to $\mathcal{M}^*$, the forcing relations for $\P$ and $\P^*$ are analogously defined in $\mathcal{M}$ and $\mathcal{M}^*$, respectively, and it follows by elementarity that $p\forces _{\P^*}\varphi(\tau)$ for forcing over $\mathcal{M}^*$. Since $p\in G^*$ this means $\<M^*[G^*],\in,A^*,G^*>\models \varphi[\tau_{G^*}]$, as desired since $\tau_G=\tau_{G^*}$.
\end{proof}

If $\P\of M$ is a definable \emph{chain of complete subposets}\footnote{A class partial order $\P$ is a  \emph{chain of complete subposets} if there is a class $\<\P_\xi\st\xi<\ORD>$ of partially ordered sets $\P_\xi $ such that $\P=\bigcup_{\xi\in\ORD}\P_\xi$ and $\P_\xi$ is a complete subposet of $\P_\eta$ whenever $\xi\leq\eta$. See for instance~\cite{Reitz2006:Dissertation}.} in some transitive model $M\models\ZFC$, then it is a standard result in the theory of class forcing that $\P$ is nice for class forcing over $M$, and that the niceness of $\P$ is preserved to $M^*$ whenever \hbox{$M\elesub M^*$} for some transitive model $M^*$. Consequently, lemma~\ref{L.LiftgStrUplift} is widely applicable as many class partial orderings can be written as a chain of complete subposets, and it applies for instance to the special case when the partial order $\P$ is an $\ORD$-length forcing iteration in $M$, since every initial part of the iteration embeds completely into the later stages.

\begin{theorem} \label{T.UpliftImplyRA(prop)}
If $\kappa$ is an uplifting cardinal, then the countable-support lottery iteration of proper forcing, defined relative to a Menas function $f\from\kappa\to\kappa$, forces $\RA(\proper)$ and $\continuum=\kappa=\aleph_2$.
 \end{theorem}

\begin{proof}
Suppose that $\kappa$ is uplifting and $f\from\kappa\to\kappa$ is a function with the Menas property for $\kappa$, such as the function of theorem~\ref{T.UpliftImplyMenas}. Let $\P$ be the countable-support lottery iteration of proper forcing defined relative to $f$. That is, $\P$ is the countable-support $\kappa$-iteration, where the forcing at stage $\beta\in\dom(f)$ is the lottery sum in $V[G_\beta]$ of all proper posets in $H_{f(\beta)^+}^{V[G_\beta]}$. We defined the iteration $\P$ in $V$ relative to $f$, but as $\kappa$ is inaccessible it follows by absoluteness that $\P$ is the same as the corresponding class lottery iteration of proper posets, relative to $f$, as defined in $\<H_\kappa,f>$. Since initial stages of $\P$ completely embed into later stages, it follows that $\P$ is a definable chain of complete subposets in $\<H_\kappa,f>$. Since the lottery sum of any number of proper forcing notions is still proper, it follows that $\P$ is a countable-support iteration of proper posets and therefore is itself proper. A standard $\Delta$-system argument shows that $\P$ is $\kappa$-c.c. Suppose that $G\of\P$ is $V$-generic. A simple density argument shows that $\kappa$ becomes $\omega_2^{V[G]}$, because both $\aleph_1$ and $\kappa$ are preserved, but all cardinals of $V$ between $\aleph_1$ and $\kappa$ have plenty of opportunity to be collapsed. Similarly, $\kappa=\continuum$ in $V[G]$, since the generic filter will opt to add reals at unboundedly many stages of the forcing. Thus $\kappa=\continuum=\aleph_2$ in $V[G]$, and it remains to prove that $V[G]\satisfies\RA(\proper)$.

Suppose that $\Q$ is any proper notion of forcing in $V[G]$, and let $\Qdot$ be a name for $\Q$ that necessarily yields a proper poset. Since $f$ has the Menas property for the uplifting cardinal $\kappa$, there is an inaccessible cardinal $\gamma$  above $\kappa$ such that $\<H_\kappa, f > \elesub \< H_\gamma, f^*> $ with $f^*(\kappa)\geq |\text{trcl}(\Qdot)|$. Thus, the poset $\Q$ appears in the stage $\kappa$ lottery of the corresponding countable-support class lottery iteration $\P^*$ of proper posets, relative to $f^*$, as defined in $\<H_\gamma,f^*>$. Notice that $\P$ and $\P^*$ agree on the stages below $\kappa$, and so below a condition opting for $\Q$ at stage $\kappa$, we may factor $\P^*$ as $\P*\Qdot*\Ptail$. Let $g*h\of\Q*\Ptail$ be $V[G]$-generic. It follows that $G*g*h$ generates a $V$-generic filter $G^*\of\P^*$. Since $G$ and $G^*$ agree on the first $\kappa$ many stages, it follows by lemma~\ref{L.LiftgStrUplift} that $H_\kappa\elesub H_\gamma$ lifts to $H_\kappa[G]\elesub H_\gamma[G^*]$. Since $\P$ is $\kappa$-c.c. and $\kappa$ is regular, we may use nice names for bounded subsets of $\kappa$ to see that $H_\kappa[G]={H_\kappa}^{V[G]}$. Since $\kappa=\continuum$ in $V[G]$, it follows in summary that $H_\kappa[G]= \Hc^{V[G]}$. Arguing analogously for the poset $\P^*$ and the regular cardinal $\gamma$, we see that $H_\gamma[G^*]=H_\gamma^{V[G^*]}=\Hc^{V[G^*]}$ and the desired elementarity $\Hc^{V[G]} \elesub \Hc^{V[G][g*h]}$ follows. Lastly, $\Ptail$ is a countable-support class iteration of proper posets in $H_\gamma[G*g]=H_\gamma^{V[G*g]}$, and consequently a proper poset in $V[G*g]$ by absoluteness. This completes the proof.
\end{proof}

The method of proof in theorem~\ref{T.UpliftImplyRA(prop)} is flexible and can be applied to many classes $\Gamma$ of forcing notions, as long as lottery sums of posets in $\Gamma$ are themselves in $\Gamma$, and a suitable preservation theorem holds for iterations of posets in $\Gamma$. For example, we obtain the following:

\begin{theorem} \label{T.UpliftImplyManyRAs}
Let $\kappa$ be an uplifting cardinal and $f\from\kappa\to\kappa$ a function with the uplifting Menas property for $\kappa$. Then
\begin{enumerate}
\item the countable-support lottery iteration of axiom-A forcing, relative to $f$, forces $\RA(\text{axiom-A})$ and $\continuum=\kappa=\aleph_2$.
\item the revised-countable-support lottery iteration of semi-proper forcing, relative to $f$, forces $\RA(\semiproper)$ and $\continuum=\kappa=\aleph_2$.
\item the finite-support lottery iteration of all forcing, relative to $f$, forces $\RAall$ and $\continuum=\kappa=\aleph_1$.
\end{enumerate}
\end{theorem}

\begin{proof}
For~(1), note that the lottery sum of any number of axiom-A posets continues to have axiom-A, and Koszmider's result~\cite{Koszmider1993:CoherentFamilies} shows that a countable-support iteration of axiom-A forcing notions has itself axiom-A. It is then straightforward to follow the proof of theorem~\ref{T.UpliftImplyRA(prop)} closely to obtain statement~(1).

For~(2), note that the lottery sum of semi-proper posets is still semi-proper, and semi-properness is preserved under iterations with revised countable support, by a result due to Shelah. Replacing the countable-support lottery iteration of proper posets in the proof of theorem~\ref{T.UpliftImplyRA(prop)} by a revised-countable-support lottery iteration of semi-proper posets therefore proves statement~(2).

For statement (3), to produce a forcing extension $V[G]$ that satisfies the resurrection axiom $\RAall$, we allow every poset at each stage of the lottery iteration. Thus, following the proof of theorem~\ref{T.UpliftImplyRA(prop)} but using a finite-support lottery iteration of all posets shows that the uplifting cardinal $\kappa$ will turn into $\aleph_1^{V[G]}$, since $\kappa$ is preserved but all uncountable cardinals below $\kappa$ are collapsed, and statement~(3) follows.
\end{proof}

We would like to call attention to the fact that the poset used in statement~(3) of theorem~\ref{T.UpliftImplyManyRAs} ---the finite-support lottery iteration of all posets, relative to a function $f\from\kappa\to\kappa$ with the uplifting Menas property---has size $\kappa$, is $\kappa$-c.c., and necessarily collapses all cardinals below $\kappa$ to $\omega$. Standard arguments show that this poset is hence forcing equivalent to the \Levy\ collapse $\Coll(\omega,\ltkappa)$, and an alternative formulation of~(3) would state that the \Levy\ collapse $\Coll(\omega,\ltkappa)$ of an uplifting cardinal~$\kappa$ forces $\RAall$.

For countable ordinals $\alpha$, Shelah~\cite{Shelah:ProperForcingLectureNotes1982}, introduced the notion of an $\alpha$-proper poset as a strengthening of properness, and he showed that $\alpha$-properness is preserved under countable-support iterations. Analogously to the other examples in theorem~\ref{T.UpliftImplyManyRAs}, it follows that the countable-support lottery iteration of $\alpha$-proper posets, relative to a function $f\from\kappa\to\kappa$ with the Menas property for $\kappa$,  forces $\RA(\alpha\text{-proper})$ with $\continuum=\kappa=\aleph_2$.

The lottery iteration method does not seem to work directly in the case of c.c.c.~forcing, simply because an uncountable lottery sum of c.c.c.~forcing is no longer c.c.c. If one were to try to perform a finite-support lottery preparation of c.c.c.~forcing, the iteration itself would no longer be c.c.c., and indeed it would collapse $\omega_1$. The argument would break down in the step where, in attempt to verify $\RA(\ccc)$, we seek to use the tail forcing $\Ptail$ as the resurrection $\Rdot$ forcing, since this would not be c.c.c.~as required.

Nevertheless, we can carry out a modified argument, using the method of Laver functions rather than lottery sums. Suppose that $\kappa$ is a cardinal, $\ell\from\kappa\to V_\kappa$ is a function, and $\Gamma$ is a class of forcing notions. We say that a forcing notion $\P$ is a \emph{Laver-style $\kappa$-iteration of\/ $\Gamma$ forcing}, defined relative to $\ell$, if $\P$ is a forcing iteration of length $\kappa$ (with some specified support) that forces at stage $\beta\in\dom(\ell)$ with $\ell(\beta)$, provided that this is a $\P_\beta$-name for forcing that is forced to be in $\Gamma^{V[G_\beta]}$; the forcing is trivial at stages $\beta\notin\dom(\ell)$. Note that any uplifting cardinal is uplifting in $L$, and by theorem~\ref{T.VariousUpliftingLaverFcns} has an uplifting Laver function there, as in the hypothesis in the following theorem.

\begin{theorem}\label{T.UpliftGivesRAccc}
 If $\kappa$ is uplifting and $\ell\from \kappa \to H_\kappa$ is an uplifting Laver function, then the finite-support Laver-style $\kappa$-iteration of c.c.c.~forcing defined relative to $\ell$ forces $\RA(\ccc)$ with $\continuum=\kappa$.
\end{theorem}

\begin{proof}
Suppose that $\kappa$ is uplifting and $\ell\from\kappa\to H_\kappa$ is an uplifting Laver function for $\kappa$. Let $\P$ be the corresponding finite-support Laver-style $\kappa$-iteration of c.c.c.~forcing defined relative to $\ell$. Thus, $\P$ forces at stage $\beta$ with $\ell(\beta)$, provided that $\ell(\beta)$ is a $\P_\beta$-name that necessarily yields a c.c.c.~poset in $V[G_\beta]$. Suppose that $G\of\P$ is $V$-generic, and consider the model $V[G]$. If $\Q$ is some c.c.c.~forcing in $V[G]$, let $\Qdot\in V$ be a $\P$-name for $\Q$, forced to be c.c.c. By the Laver function property, there is an inaccessible cardinal $\gamma$ for which $\<H_\kappa,\ell>\elesub \<H_\gamma,\ell^*>$ and $\ell^*(\kappa)=\Qdot$. Note that the definition of $\P$ works inside $\<H_\kappa,\ell>$, and so we may let $\P^*$ be the corresponding iteration as defined in $\<H_\gamma,\ell^*>$. The forcing notions $\P$ and $\P^*$ agree on the stages below $\kappa$, and since $\ell^*(\kappa)=\Qdot$ is a $\P$-name forced to be c.c.c.,~ the poset $\P^*$ forces at stage $\kappa$ with $\Q$, and $\P^*$ factors as $\P*\Qdot*\Ptail$, where $\Ptail$ is the forcing after stage $\kappa$. Suppose that $g*h\of\Q*\Ptail$ is $V[G]$-generic, so that $G*g*h$ is $V$-generic for $\P^*$. Thus, by lemma~\ref{L.LiftgStrUplift} we may lift $H_\kappa\elesub H_\gamma$ to $H_\kappa[G]\elesub H_\gamma[G][g*h]$. Since $\P$ is c.c.c.~and $\kappa=\continuum$ in $V[G]$, it follows that $H_\kappa[G]=\Hc^{V[G]}$; an analogous argument for $\P^*$ shows that $H_\gamma[G][g*h]=\Hc^{V[G][g][h]}$ and consequently $\Hc^{V[G]}\elesub\Hc^{V[G][g*h]}$. Since $\Ptail$ is a finite support iteration of c.c.c.~forcing in $V[G][g]$, it is c.c.c., and so we have established $\RA(\ccc)$ in $V[G]$.
\end{proof}

The method of proof in theorem~\ref{T.UpliftGivesRAccc} is general, and it applies, for example, to the classes $\Gamma$ considered in theorems~\ref{T.UpliftImplyRA(prop)} and~\ref{T.UpliftImplyManyRAs}, providing alternative proofs of those theorems. For instance, if  $\kappa$ is uplifting and $\LDuplift_\kappa$ holds at $\kappa$, then the countable-support Laver-style $\kappa$-iteration of proper forcing, defined relative to a corresponding Laver function, forces $\RA(\proper)$ with $\continuum=\kappa=\aleph_2$.

In summary, the previous theorems establish our main result:

\begin{maintheorem}\label{T.Maintheorem}
The following theories are equiconsistent over $\ZFC$:
\begin{enumerate}
\item There is an uplifting cardinal.
\item $\RAall$
\item $\RA(\ccc)$
\item $\RA(\semiproper)+\neg\CH$
\item $\RA(\proper)+\neg\CH$
\item for some countable ordinal $\alpha$, $\;\RA(\alpha\text{-proper})+\neg\CH$
\item $\RA(\text{axiom-A})+\neg\CH$
\item $\wRA(\semiproper)+\neg\CH$
\item $\wRA(\proper)+\neg\CH$
\item for some countable ordinal $\alpha$,  $\;\wRA(\alpha\text{-proper})+\neg\CH$
\item $\wRA(\text{axiom-A})+\neg\CH$
\item $\wRA(\text{countably closed})+\neg\CH$
\end{enumerate}
\end{maintheorem}

Note that the axiom $\RA(\text{preserving stationary subsets of }\omega_1)+\neg\CH$ is not mentioned in the theorem, and neither is its weakening to the axiom $\wRA(\text{preserving stationary subsets of }\omega_1)+\neg\CH$. The reason is that each of these axioms implies $\BMM$ by theorem~\ref{T.wRAimpliesBFA_kappa}, but $\BMM$ has much higher consistency strength than an uplifting cardinal, as it implies the existence of an inner model with a strong cardinal, by a result due to Schindler~\cite{Schindler2006:BMMandStrong}. We prove the relative consistency of $\RA(\text{preserving stationary subsets of }\omega_1)+\neg\CH$ in the next section in theorem~\ref{T.RAproper+PFAVariants}.

\section{Resurrection axioms and $\PFA$, $\SPFA$, $\MM$, and their fragments}\label{S.RAandPFA}

The methods of the previous section lend themselves to be combined with standard techniques to produce models of the proper forcing axiom $\PFA$ and the semi-proper forcing axiom $\SPFA$, starting with a supercompact cardinal, as well as with techniques used by Tadatoshi Miyamoto in~\cite{Miyamoto1998:WeakSegmentsOfPFA} and also by the authors of this article in~\cite{HamkinsJohnstone2009:PFA(aleph_2-preserving)}  to produce models of certain fragments of $\PFA$ and $\SPFA$, starting with a strongly unfoldable cardinal. We will also obtain instances of resurrection axioms in models of Martin's Maximum \MM, or of the axiom-A forcing axiom $\text{AAFA}$. 

Suppose that $\kappa$ is a supercompact cardinal. By Laver's result~\cite{Laver78}, there is a supercompactness Laver function $\ell\from\kappa\to H_\kappa$. Laver's function does not seem in general to be definable in $H_\kappa$, although one can find such a definable Laver function when $H_\kappa$ has a definable well-ordering. In the general case, we shall consider instead the \emph{failure-of-supercompactness} function $f\from\kappa\to\ORD$, which maps every non-supercompact cardinal $\gamma<\kappa$ to the least $\lambda$ such that $\gamma$ is not $\lambda$-supercompact. It is easy to see that $f$ is definable in $H_\kappa$ and that $f$ is a supercompactness {\df Menas} function, meaning that for every ordinal $\theta$ there is an elementary embedding $j:V\to M$ with critical point $\kappa$ and $M^\theta\of M$ and $j(\kappa)\geq\theta$ such that $j(f)(\kappa)\geq \theta$.

Suppose now that $\kappa$ is a supercompact cardinal that is also uplifting. Since this implies the existence of many large cardinals above $\kappa$, the overall consistency strength of this hypothesis is strictly above that of a supercompact cardinal, although the proof of theorem \ref{T.ConBoundUplift} shows that it is bounded above by a stationary set of supercompact cardinals. Let $f_1\from\kappa\to\kappa$ be the failure-of-supercompactness function as discussed above, and $f_2\from\kappa\to\kappa$ be any uplifting Menas function, such as the failure-of-upliftingness function as in theorem~\ref{T.UpliftImplyMenas}, and let $f=\max(f_1,f_2)$. Since $f_1$ is definable in $H_\kappa$, it follows that $f\from\kappa\to\kappa$ is a Menas function both for upliftingness and for supercompactness, exactly as is needed for the next theorem.

\begin{theorem}\label{T.RAproper+PFAVariants}
 Suppose that $\kappa$ is both supercompact and uplifting and that $f\from\kappa\to\kappa$ is a function with the Menas property for both supercompactness and upliftingness. Then
 \begin{enumerate}
 \item the countable-support lottery iteration of proper forcing, relative to $f$, forces $\RA(\proper)+\PFA$.
 \item the countable-support lottery iteration of axiom-A forcing, relative to $f$, forces $\RA(\text{axiom-A})+\text{\rm AAFA}$.
\item the revised-countable-support lottery iteration of semi-proper forcing, relative to $f$, forces $\RA(\text{preserving stationary subsets of }\omega_1)+ \RA(\semiproper)+\MM.$
 \end{enumerate}
\end{theorem}

\begin{proof} For (1), the resurrection axiom $\RA(\proper)$ holds by theorem~\ref{T.UpliftImplyRA(prop)}, and the verification of $\PFA$ is essentially the Baumgartner argument, but using the lottery iteration in place of his Laver-style iteration (for a proof see~theorem~12 in \cite{HamkinsJohnstone2009:PFA(aleph_2-preserving)}). The proof of statement~(2) is essentially the same, using theorem~\ref{T.UpliftImplyManyRAs}. To prove statement (3), note that the forcing extension satisfies both $\SPFA$ and $\RA(\semiproper)$, by the analogous argument of (1) applied to the semi-proper case. Since Shelah~\cite{Shelah1987:SPFAimpliesMM} showed that $\SPFA$ implies that every poset which preserves stationary subsets of $\omega_1$ is semi-proper, it follows that $\MM$ if and only if $\SPFA$ and consequently that $\RA(\text{preserving stationary subsets of }\omega_1)$ holds in the forcing extension, as desired.
\end{proof}

The same method shows for any $\alpha<\omega_1$ that the countable-support lottery iteration of $\alpha$-proper forcing, relative to a function $f\from\kappa\to\kappa$ with the Menas property for both supercompactness and upliftingness, forces $\RA(\alpha\text{-proper})$ and the forcing axiom $\FA(\alpha\text{-proper})$.

Strongly unfoldable cardinals were introduced by Villaveces~\cite{Villaveces1998:ChainsOfEndElementaryExtensionsOfModelsOfSetTheory} to exhibit a miniature form of strongness, and they were introduced independently by Miyamoto in~\cite{Miyamoto1998:WeakSegmentsOfPFA} as the \emph{$H_{\kappa^{+}}$-reflecting} cardinals by an equivalent characterization exhibiting a miniature form of supercompactness (see also \cite{Hamkins2001:UnfoldableCardinals,DzamonjaHamkins2006:DiamondCanFail}). They lie relatively low in the large cardinal hierarchy, somewhat above the weakly compact cardinals; their consistency strength is bounded below by the totally indescribable cardinals and above by the subtle cardinals, and they relativize to $L$. One of the main results in \cite{CodyGitikHamkinsSchanker:LeastWeaklyCompact} shows that the least weakly compact cardinal can be unfoldable. For a detailed account of strongly unfoldable cardinals and their indestructibility properties, we may refer the reader to our article~\cite{HamkinsJohnstone2010:IndestructibleStrongUnfoldability}, and for a quick review of these cardinals to~\cite{HamkinsJohnstone2009:PFA(aleph_2-preserving)}. There, we used a strongly unfoldable cardinal $\kappa$ and a countable-support lottery iteration of proper forcing, which we called the {\df $\PFA$ lottery preparation}, to establish the relative consistency of several fragments of $\PFA$ and $\SPFA$. Just as with supercompact cardinals, the \emph{failure-of-strong-unfoldability} function $f\from\kappa\to\kappa$ for a strongly unfoldable cardinal $\kappa$ has the strong-unfoldability Menas property for $\kappa$, and it is definable in $H_\kappa$ (see details in \cite{HamkinsJohnstone2009:PFA(aleph_2-preserving)}).

Following~\cite{HamkinsJohnstone2009:PFA(aleph_2-preserving)}, let us review the relevant concepts used in theorem~\ref{T.PFACoveringFragmentsWithRA}. If $\Gamma$ is a class of posets, than the forcing axiom $\PFA(\Gamma)$ is the assertion that for any proper poset $\Q\in\Gamma$ and every collection $\mathcal{A}$ of at most $\aleph_1$ many maximal antichains in $\Q$, there is a filter on $\Q$ meeting each antichain in $\mathcal{A}$.  If $\delta$ is a cardinal, then the forcing axiom $\PFA_\delta$ is the assertion that for any proper complete Boolean algebra $\B$ and any collection $\mathcal{D}$ of at most $\aleph_1$ many maximal antichains in $\B\minus\{0\}$, each antichain of size at most $\delta$, there is a filter on $\B$ meeting each antichain in $\mathcal{A}$. A forcing notion $\Q$ is \hbox{\emph{$\delta$-preserving}} if forcing with $\Q$ does not collapse $\delta$ as a cardinal. A forcing notion $\Q$ is \emph{$\delta$-covering} if whenever $G\of \Q$ is $V$-generic and $A\in V[G]$ is a set of ordinals with $|A|^{V[G]}<\delta$, then there is a cover $B\in V$ such that $A\of B$ and $|B|^V<\delta$. Note that for any cardinal $\delta$, every $\delta$-covering forcing notion is necessarily $\delta$-preserving.

\begin{theorem}\label{T.PFACoveringFragmentsWithRA}
 Suppose that $\kappa$ is strongly unfoldable and let $\P$ be the countable-support lottery iteration of proper posets, relative to a function $f\from\kappa\to\kappa$ with the Menas property for strong unfoldability. Then:
 \begin{enumerate}
 \item If $\kappa$ is uplifting and $f$ has the uplifting Menas property, then $\P$ forces $\RA(\proper)+\PFA(\aleph_2\text{-covering})+\PFA(\aleph_3\text{-covering})+\PFA_{\continuum}$ and $\continuum=\kappa=\aleph_2$. If $0^\sharp$ does not exist, then $\P$ forces the additional axioms $\PFA(\aleph_2\text{-preserving})$ and $\PFA(\aleph_3\text{-preserving})$.
\item If $\kappa$ is not uplifting in $L$, then $\P$ forces $\neg\wRA(\text{countably closed})+\PFA(\aleph_2\text{-covering})+\PFA(\aleph_3\text{-covering})+\PFA_{\continuum}$ with $\continuum=\kappa=\aleph_2$. If $0^\sharp$ does not exist, then $\P$ forces the additional axioms $\PFA(\aleph_2\text{-preserving})$ and $\PFA(\aleph_3\text{-preserving})$.
 \end{enumerate}
\end{theorem}

\begin{proof}
The verification of all relevant fragments of $\PFA$ in statements~(1) and~(2) is exactly what is proved in theorems~3,~4 and~6 of~\cite{HamkinsJohnstone2009:PFA(aleph_2-preserving)}. The resurrection axiom $\RA(\proper)$ with $\continuum=\kappa=\aleph_2$ holds in statement~(1) by theorem~\ref{T.UpliftImplyRA(prop)}, but $\wRA(\text{countably closed})$ must fail in statement~(2), since otherwise $\kappa$ would be uplifting in $L$ by theorem~\ref{T.RAimplyCUplift}, a contradiction.\end{proof}

The hypothesis of statement~(1) of theorem~\ref{T.PFACoveringFragmentsWithRA} is equiconsistent with a strongly unfoldable uplifting cardinal, since every such cardinal is strongly unfoldable and uplifting in $L$, and so we may work in $L$, where $0^\sharp$ does not exist, if necessary. The hypothesis of statement~(2) of theorem~\ref{T.PFACoveringFragmentsWithRA} is equiconsistent with a strongly unfoldable cardinal, since if $\kappa$ is strongly unfoldable, then it remains so in $L$, and we may work again in $L$ and chop off the universe at the first inaccessible cardinal above $\kappa$, if necessary, which results in a strongly unfoldable cardinal in $L$ that is not uplifting there, as desired.

Results analogous to theorem~\ref{T.PFACoveringFragmentsWithRA} hold for axiom-A forcing, for $\alpha$-proper forcing with $\alpha<\omega_1$, and for semi-proper forcing by essentially the same proofs. If the existence of a strongly unfoldable uplifting cardinal is consistent, then it follows from theorem~\ref{T.PFACoveringFragmentsWithRA} that $\RA(\proper)$ is independent from the conjunction of the three forcing axioms $\PFA_{\continuum}$, $\PFA(\aleph_2\text{-preserving})$ and $\PFA(\aleph_3\text{-preserving})$. The same holds for the weak resurrection axioms, such as $\wRA(\proper)$ and $\wRA(\text{countably closed})$, and also for the axioms $\RA(\text{semi-proper})$ and $\RA(\text{axiom-A})$ and their weak counterparts $\wRA(\text{semi-proper})$ and $\wRA(\text{axiom-A})$. Meanwhile, it seems unclear to us how to show, say, that $\RA(\proper)$ can fail if $\PFA$ holds.

We conclude this paper by foreshadowing our follow-up article \cite{HamkinsJohnstone:BoldfaceResurrectionAndSuperstrongUnfoldability}, a natural continuation of this article, in which we introduce and consider the boldface analogues of the resurrection axioms, allowing a predicate $A\of\continuum$ and asking for $A^*\of\continuum^{V[g*h]}$ in $V[g*h]$ with $\<\Hc,{\in},A>\elesub\<\Hc^{V[g*h]},{\in},A^*>$. In that article, we prove the equiconsistency of the boldface resurrection axioms with the existence of a strongly uplifting cardinal, a weak form of $1$-extendibility, which we prove is the same as a superstrongly unfoldable cardinal, generalizing the weakly superstrong cardinals.

\bibliographystyle{alpha}
\bibliography{MathBiblio,HamkinsBiblio}

\end{document}